\documentclass[12pt]{amsart}
\usepackage{amsfonts}
\usepackage{amsmath}
\usepackage{amssymb}
\usepackage{graphicx}
\usepackage{enumerate}
\usepackage{multicol}
\usepackage{mathrsfs}
\usepackage[usenames,dvipsnames]{pstricks}
\usepackage{epsfig}
\usepackage[hmargin=3cm,vmargin=2.5cm]{geometry}
\usepackage{lscape}
\usepackage{epsfig}
\usepackage{verbatim}
\usepackage[utf8]{inputenc}
\usepackage{tikz}
\usepackage{mathdots}

\newtheorem{proposition}{Proposition}
\newtheorem{theorem}{Theorem}

\newtheorem{coro}{Corollary}
\newtheorem{lemma}{Lemma}

\DeclareMathOperator{\Cl}{Cl}

\DeclareMathOperator{\spn}{span}

\newcommand{\g}{\mathfrak{g}}
\newcommand{\h}{\mathfrak{h}}
\newcommand{\p}{\mathfrak{p}}
\newcommand{\n}{\mathfrak{n}}

\DeclareMathOperator{\Hom}{\rm{Hom}}

\title[Lie algebras attached to Clifford modules]
{Lie algebras attached to Clifford modules and simple graded Lie algebras}
\author[Furutani, Godoy Molina, Markina, Morimoto, Vasil'ev]{Kenro Furutani$^{\diamondsuit}$, Mauricio Godoy Molina$^{\dag\,\ast}$, Irina Markina$^{\dag,\ddag}$,\\Tohru Morimoto, Alexander Vasil'ev$^{\dag,\ddag}$}

\address{K. Furutani: Department of Mathematics, Science University of Tokyo, Japan}
\email{furutani\_kenro@ma.noda.tus.ac.jp}
\address{M. Godoy Molina: Departamento de Matem\'atica y Estad\'istica, Universidad de la Frontera, Chile.}
\email{mauricio.godoy@ufrontera.cl}
\address{I. Markina, A. Vasil'ev: Department of Mathematics, University of Bergen, Norway.}
\email{irina.markina@uib.no, alexander.vasiliev@uib.no}
\address{T. Morimoto: Oka Mathematical Institute, Nara Women's University, Japan.}
\email{morimoto@cc.nara-wu.ac.jp}
\thanks{The first author${^{\diamondsuit}}$ is partially
supported by the Grant-in-aid for
Scientific Research (C) No. 26400124, 
Japan Society for the Promotion of Science. The author$^{\ast}$ is partially supported by Grant DI17-0147 of the Universidad de La Frontera, Chile. 
The authors$^{\dag}$ are partially supported by the grant of the Chilean Research Council Anillo ACT 1415 PIA CONICYT and EU FP7 IRSES program STREVCOMS, grant no. PIRSES-GA-2013-612669. The authors$^{\ddag}$ are partially supported by the grant of the Norwegian Research Council \#239033/F20.}



\begin{document}

\maketitle

\begin{abstract}
We study possible cases of complex simple graded Lie algebras of depth 2, which are the Tanaka prolongations of pseudo $H$-type Lie algebras arising through representation of Clifford algebras. We show that the complex simple Lie algebras of type $B_n$ with $|2|$-grading do not contain non-Heisenberg pseudo $H$-type Lie algebras as their negative nilpotent part, while the complex simple Lie algebras of types $A_n$, $C_n$ and $D_n$ provide such a possibility. Among exceptional algebras only $F_4$ and $E_6$ contain non-Heisenberg pseudo $H$-type Lie algebras as their negative part of $|2|$-grading. An analogous question addressed to real simple graded Lie algebras is more difficult, and we give results revealing the main differences with the complex situation.
\end{abstract}

\maketitle

\section{Introduction}

Given a finite dimensional nilpotent graded Lie algebra ${\mathfrak n}=\bigoplus_{p < 0}{\mathfrak n} _p$,
there is a uniquely associated maximal transitive graded Lie algebra 
${\mathfrak g} = \bigoplus _ { p \in \mathbb Z } \mathfrak g _p$,
such that the negative part 
$\mathfrak{g}_{-}  =\bigoplus_{p < 0} \mathfrak{g}_p$ of $\mathfrak g $ 
coincides with $\mathfrak{n}$.
This graded Lie algebra $\mathfrak{g}$ called the {\it Tanaka prolongation} may be regarded as the symmetry algebra of $\n$ and
plays a fundamental role in the study of geometric structures associated with 
differential systems of type $\mathfrak{n}$, see~\cite{Ta}. 
In particular, if the Tanaka prolongation is 
finite dimensional and simple, 
or more generally, if it satisfies the so-called condition (C) following \cite{Mo},   
one can pursuit more detailed studies by constructing Cartan connections, see \cite{Mo, Mo2,Ta2}.

A special class of nilpotent Lie algebras, called
$H$(eisenberg)-type algebras was introduced
by Kaplan in~\cite{Kaplan}. This class is
associated with the Clifford algebra generated by a vector space with
a positive definite quadratic form. A generalization related to the Clifford algebras with 
arbitrary non-degenerate quadratic forms was proposed
in~\cite{Ciatti, GKM} and received the name {\it pseudo $H$-type Lie algebras}.
We are interested in studying the Tanaka prolongation of these Lie
algebras and the geometric features of the related Lie groups.

There is a series $\{ \n^{r,s} \}$ of {\it basic} pseudo $H$-type algebras 
defined for integers $r, s \geq 0$, where 
$\n^{r,s}  =   \n_{-2} \oplus  \n_{-1} $ is a nilpotent graded Lie algebra ,  
 $ \n_{-2}  $ is endowed with a scalar product $\langle .\,,.\rangle_{r,s}$ of signature $(r, s)$,
 and 
  $\n_{-1} $ is a minimal admissible module of the Clifford algebra $\Cl ( \n_{-2},\langle .\,,.\rangle_{r,s})$. The prolongations for the basic pseudo $H$-type Lie algebras $ \n^{r,s}$, $r+s\leq 8$, were computed
by using the Maple software [Tanaka Prolongation] elaborated by Anderson~\cite{{AnI}} and the tables of structure constants 
of pseudo $H$-type Lie algebras~\cite{FM3}. The result, presented in Tables~\ref{T8} and ~\ref{T9} in Section~\ref{exAppendix},
leads us to a conjecture
that the prolongation $\rm{Prol}(\n)$
of a pseudo $H$-type Lie algebra $\n$ vanishes at order one if it does not contain a simple graded Lie algebra of depth 2. In other words the pseudo $H$-type Lie algebra has a simple graded Lie algebra of depth two as a symmetry algebra, (not involving metric), otherwise it is rigid. Inspired by this work, it was shown in~\cite{GKMV} 
that any pseudo $H$-type algebra $\n$ with the dimension of 
the centre $\n_{-2}$ strictly greater than 2 has finite Tanaka prolongation, see also~\cite{R}.
On the other hand it is known that the prolongation of a subriemannian symbol 
vanishes at order one~\cite{Mo3}. 

We pose the following question:

\begin{description}
\item[Problem 1] Find the pseudo $H$-type Lie algebras $\mathfrak{n}$, for which the Tanaka prolongation is a simple graded Lie algebra.  Equivalently, we can look for a simple graded Lie algebra $\mathfrak g$ such that its negative part
$\mathfrak g _{-}$ in $|2|$-grading is a pseudo $H$-type
algebra.
\end{description}

The problem of relative Tanaka prolongations of $H$-type Lie algebras with positive definite quadratic forms was treated in~\cite{KS}. 
The same classification question has been raised in the study of prolongations of (super-)Poincar\'e algebras, 
see~\cite{AC,AS1,AS2}. These results include some of ours as a particular case, since pseudo $H$-type algebras are contained in the wider class of so-called extended translation algebras. The fact that the class of pseudo $H$-type algebras is strictly smaller, allows our results to be more precise than those obtained in~\cite{AS2}.

The Clifford representations that yield pseudo $H$-type Lie algebras, 
so-called admissible Clifford modules, impose some restrictions 
on the dimensions of the pseudo $H$-type algebras. These dimensions 
are related by the Radon-Hurwitz-Eckmann formula.
In our present work we study Problem 1 mainly restricting ourselves to the following two problems:
\begin{description}
\item[Problem 2] Among all $|2|$-gradings $\g_{-2}\oplus\cdots\oplus\g_2$ of a complex simple Lie algebra $\g$ find those whose negative nilpotent part $\g_-=\g_{-2}\oplus\g_{-1}$ has the dimensions of $\g_{-2}$ and $\g_{-1}$ coinciding with those of the complexification of a pseudo $H$-type Lie algebra (no metric is involved).

\item[Problem 3] Among all candidates found in Problem 2, determine 
all real pseudo $H$-type Lie algebras whose complexification coincides 
with the nilpotent part $\g_{-2}\oplus\g_{-1}$ of the gradings of the complex simple Lie algebras. 

\end{description}

We find necessary conditions
for a $|2|$-grading $\g_{-2}\oplus\cdots\oplus\g_2$ of a complex simple Lie algebra $\g$ 
for the dimensions of $\g_{-1}$ and $\g_{-2}$ matching 
the dimensions of a complexified pseudo $H$-type algebra, by making use of the Radon-Hurwitz-Eckmann formula, as studied in~\cite{FM, FM2, FMV}. This way we fix Problem 2. We show that the classical complex Lie algebras of type $A_n$, $C_n$ and $D_n$ potentially may contain infinitely many pseudo $H$-type Lie algebras as a negative nilpotent part in their $|2|$-grading by checking carefully the necessary conditions according to each type. At the same time, a classical complex Lie algebra of type $B_n$ with any $|2|$-grading never contains any pseudo $H$-type Lie algebra as its negative nilpotent part except for the Heisenberg algebra. 

We also present some progress in solving Problem 3. We determine all real simple graded  Lie algebras whose negative part is a basic pseudo $H$-type algebra $\n^{r,s} $ for $r+s \leq 8$. As a typical example we show that the negative part $\g_{-2}\oplus\g_{-1}$ of the grading of $\mathfrak{su}(3,3)$ is isomorphic to the pseudo $H$-type algebra $\n=\n_{-2}\oplus \n_{-1}$ carrying a scalar product of index $(1,3)$ on the center $\n_{-2}$. On the other hand we prove that the only $H$-type Lie algebras with a positive definite bilinear form that appear as the negative part of $|2|$-gradings of ${\mathfrak{sl}}(n+1,{\mathbb R})$ are the Heisenberg algebras, see Theorem~\ref{th:positive}. 

For the exceptional simple Lie algebras we prove a remarkable fact that only $F_4$ and $E_6$ 
are the complexified prolongations of some pseudo $H$-type Lie algebras which are not the Heisenberg algebra, see Theorem~\ref{th:exH}. 
To prove that the exceptional simple Lie algebras $F_4$ and $E_6$ really
 appear as the symmetry algebras of a pseudo $H$-type algebra, 
we give a computer-aided proof by showing that 
the prolongation  of the basic pseudo $H$-type algebra $\n^{r,s}$ is isomorphic to
a real form  of  $F_4$  for $(r, s) = (7, 0) , (3, 4) $
and to that of  $E_6$  for $(r, s) = (8, 0) , (7, 1) , (4,4), (3, 5)$, see Table~\ref{T9}. 
We use the necessary condition obtained from the Radon-Hurwitz-Eckmann formula in order to finish the proof of the theorem.  

The structure of the present paper is the following.
After reviewing some preliminaries needed for the rest of the paper in Section~\ref{sec:prel}, we proceed presenting the complete answer to Problem 2 in Sections~\ref{sec:class} and~\ref{sec:excep}, for the classical complex Lie algebras and for the exceptional ones.  In Section~\ref{sec:example}, we show a concrete example answering affirmatively Problem 3, and in Section~\ref{sec:noHAn}, we prove that only the Heisenberg algebras appear as the negative part of $|2|$-gradings of ${\mathfrak{sl}}(n+1,{\mathbb R})$ among all $H$-type algebras endowed with a positive definite bilinear form. Finally, in Section~\ref{exAppendix} 
we list the growth vectors of prolongation $\rm{Prol}(\n^{r,s})$ for $r+s \leq 8$
computed by Maple, and then, 
identify the Lie algebras $\rm{Prol}(\n^{r,s})$ which are simple.


\section{Preliminaries}\label{sec:prel}


\subsection{Notation}\label{ssec:not}


We follow the notational conventions of~\cite{CS}. Let $\g$ be a semisimple complex Lie algebra with a Cartan subalgebra~$\h$. Denote by $\Delta$ the root system of $\g$ with respect to $\h$, and let $\g=\h\oplus\bigoplus_{\alpha\in\Delta}\g_\alpha$ be the root decomposition of~$\g$. Suppose we have chosen a set $\Delta^+$ of positive roots, thus $\Delta=\Delta^+\cup\Delta^-$, and let $\Delta^0\subset\Delta^+$ be a set of simple roots.
We say that $\p$ is a parabolic subalgebra of $\g$ if it contains the Borel subalgebra ${\mathfrak{b}}=\h\oplus\bigoplus_{\alpha\in\Delta^+}\g_\alpha$.
Given a subset of simple roots $\Sigma\subset\Delta^0=\{\alpha_1,\dotsc,\alpha_r\}$ and a root $\alpha\in\Delta$, we define the height of $\alpha=\sum_ia_i\alpha_i$ with respect to $\Sigma$ as
\[
{\rm ht}_\Sigma(\alpha)=\sum_{i\colon\alpha_i\in\Sigma}a_i.
\]
If $\alpha_{\rm max}$ is the highest root of $\g$, then we denote $k={\rm ht}_\Sigma(\alpha_{\rm max})$. We obtain a $|k|$-grading $\g=\g_{-k}\oplus\cdots\oplus\g_0\oplus\cdots\oplus\g_k$, where
\[
\g_i=\bigoplus_{\alpha\colon{\rm ht}_\Sigma(\alpha)=i}\g_\alpha,\quad i\neq0,\quad\mbox{and}\quad\g_0=\h\oplus\bigoplus_{\alpha\colon{\rm ht}_\Sigma(\alpha)=0}\g_\alpha.
\]
In the present paper we are interested in the case $k=2$ in order to study the relation to a class of 2-step nilpotent Lie algebras, which we proceed to define.

\subsection{Pseudo $H$-type algebras and Clifford modules}\label{ssec:pseudoH}

A pseudo $H$-type algebra $\n=\n_{-2}\oplus\n_{-1}$ is a real graded 2-step nilpotent Lie algebra endowed with a non-degenerate symmetric bilinear form $\langle.\,,.\rangle=\langle.\,,.\rangle_{\n_{-2}}+\langle.\,,.\rangle_{\n_{-1}}$, such that $\langle.\,,.\rangle_{\n_{-2}}$ is non-degenerate and all the maps $J_z\in{\rm End}(\n_{-1})$ defined by
\[
\langle J_zx,y\rangle_{\n_{-1}}=\langle[x,y],z\rangle_{\n_{-2}},\quad\mbox{for all }z\in\n_{-2},\quad x,y\in\n_{-1},
\]
satisfy the Clifford condition $J_z^2=-\langle z,z\rangle_{\n_{-2}}\,{\rm Id}_{\n_{-1}}$. The last identity implies that the map $J\colon\n_{-2}\to{\rm End}(\n_{-1})$ can be extended to a representation of the Clifford algebra ${\rm Cl}(\n_{-2},\langle.\,,.\rangle_{\n_{-2}})$. 

Given a representation of the Clifford algebra $\tilde J\colon{\rm Cl}(\n_{-2},\langle.\,,.\rangle_{\n_{-2}})\to{\rm End}(V)$, where $V$ is a real vector space, we say that $V$ is admissible as a Clifford module if it admits a non-degenerate scalar product $\langle.\,,.\rangle_V$ such that
\begin{equation}\label{eq:Jmap}
\langle\tilde J_zx,y\rangle_V=-\langle x,\tilde J_zy\rangle_V,\quad \mbox{for all }z\in\n_{-2},\quad x,y\in\n_{-1}.
\end{equation}
There is a strong restriction on which Clifford modules can be admissible, as noted in~\cite{Ciatti}. In particular, not all irreducible Clifford modules can be admissible; thus we define a minimal admissible module as an admissible module of minimal dimension. Note that if $\dim\n_{-2}>1$, then $\dim\n_{-1}$ must be divisible by 4.

If $\langle.\,,.\rangle_{\n_{-2}}$ as a bilinear form has $r$ positive and $s$ negative eigenvalues, and if $\n_{-1}$ is an admissible ${\rm Cl}(\n_{-2},\langle.\,,.\rangle_{\n_{-2}})$-module of minimal dimension, then we denote by $\n^{r,s}$ the pseudo $H$-type algebra $\n_{-2}\oplus\n_{-1}$.

\subsection{Choices of roots for the classical Lie algebras}\label{ssec:chooseSigma}

In this subsection we construct $|2|$-gradings of the complex simple Lie algebras using the conventions from~\cite{B}. Following the procedure described in Section~\ref{ssec:not}, the only possibilities for $|2|$-gradings are given in Table~\ref{t:dynsig}.
{\small{
\begin{table}[h]
\caption{Dynkin diagrams and choices of $\Sigma$ for $|2|$-gradings of the classical Lie algebras}\label{t:dynsig}
\begin{tabular}{|c|c|c|}
\hline
Algebra&$\Sigma$&Dynkin diagram\\
\hline
$A_n\colon{\mathfrak{sl}}(n+1,{\mathbb C})$&$\Sigma_{i,j}=\{\alpha_i,\alpha_j\}$&
\scalebox{0.5} 
{
\begin{pspicture}(0,-0.39919922)(10.88291,0.41919923)
\usefont{T1}{ptm}{m}{n}
\rput(0.7624707,0.22580078){$\alpha_1$}
\pscircle[linewidth=0.04,dimen=outer](0.78101563,-0.17919922){0.2}
\psline[linewidth=0.04cm](0.9810156,-0.17919922)(2.3810155,-0.17919922)
\pscircle[linewidth=0.04,dimen=outer](9.781015,-0.17919922){0.2}
\pscircle[linewidth=0.04,dimen=outer](7.9810157,-0.17919922){0.2}
\psline[linewidth=0.04cm](8.181016,-0.17919922)(9.581016,-0.17919922)
\psline[linewidth=0.04cm,linestyle=dotted,dotsep=0.16cm](6.381016,-0.17919922)(7.7810154,-0.17919922)
\psline[linewidth=0.04cm,linestyle=dotted,dotsep=0.16cm](4.5810156,-0.17919922)(5.9810157,-0.17919922)
\pscircle[linewidth=0.04,dimen=outer](2.5810156,-0.17919922){0.2}
\psline[linewidth=0.04cm,linestyle=dotted,dotsep=0.16cm](2.7810156,-0.17919922)(4.1810155,-0.17919922)
\usefont{T1}{ptm}{m}{n}
\rput(2.5624707,0.22580078){$\alpha_2$}
\usefont{T1}{ptm}{m}{n}
\rput(4.3224707,0.22580078){$\alpha_i$}
\usefont{T1}{ptm}{m}{n}
\rput(6.122471,0.22580078){$\alpha_j$}
\usefont{T1}{ptm}{m}{n}
\rput(8.032471,0.22580078){$\alpha_{n-1}$}
\usefont{T1}{ptm}{m}{n}
\rput(9.902471,0.22580078){$\alpha_n$}
\psline[linewidth=0.04cm](4.1810155,0.02080078)(4.5810156,-0.3791992)
\psline[linewidth=0.04cm](5.9810157,0.02080078)(6.381016,-0.3791992)
\psline[linewidth=0.04cm](6.381016,0.02080078)(5.9810157,-0.3791992)
\psline[linewidth=0.04cm](4.5810156,0.02080078)(4.1810155,-0.3791992)
\end{pspicture} 
}
\\
$n\geq2$&$1\leq i\leq\left[\frac{n}2\right]$&
\\
&$i< j\leq n+1-i$&
\\
\hline
$B_n\colon{\mathfrak{so}}(2n+1,{\mathbb C})$&$\Sigma_i=\{\alpha_i\}$&
\scalebox{0.5} 
{
\begin{pspicture}(0,-0.39919922)(10.80291,0.41919923)
\pscircle[linewidth=0.04,dimen=outer](9.781015,-0.17919922){0.2}
\pscircle[linewidth=0.04,dimen=outer](7.9810157,-0.17919922){0.2}
\pscircle[linewidth=0.04,dimen=outer](6.1810155,-0.17919922){0.2}
\psline[linewidth=0.04cm](6.381016,-0.17919922)(7.7810154,-0.17919922)
\psline[linewidth=0.04cm,linestyle=dotted,dotsep=0.16cm](4.5810156,-0.17919922)(5.9810157,-0.17919922)
\pscircle[linewidth=0.04,dimen=outer](2.5810156,-0.17919922){0.2}
\psline[linewidth=0.04cm,linestyle=dotted,dotsep=0.16cm](2.7810156,-0.17919922)(4.1810155,-0.17919922)
\usefont{T1}{ptm}{m}{n}
\rput(2.5624707,0.22580078){$\alpha_2$}
\usefont{T1}{ptm}{m}{n}
\rput(4.3224707,0.22580078){$\alpha_i$}
\usefont{T1}{ptm}{m}{n}
\rput(6.292471,0.22580078){$\alpha_{n-2}$}
\usefont{T1}{ptm}{m}{n}
\rput(8.092471,0.22580078){$\alpha_{n-1}$}
\usefont{T1}{ptm}{m}{n}
\rput(9.962471,0.22580078){$\alpha_n$}
\psline[linewidth=0.04cm](8.141016,-0.13919921)(9.621016,-0.13919921)
\psline[linewidth=0.04cm](8.141016,-0.21919923)(9.621016,-0.21919923)
\psline[linewidth=0.04cm](8.981015,-0.17919922)(8.741015,-0.059199218)
\psline[linewidth=0.04cm](8.981015,-0.17919922)(8.741015,-0.29919922)
\pscircle[linewidth=0.04,dimen=outer](0.78101563,-0.17919922){0.2}
\psline[linewidth=0.04cm](0.9810156,-0.17919922)(2.3810155,-0.17919922)
\usefont{T1}{ptm}{m}{n}
\rput(0.7624707,0.22580078){$\alpha_1$}
\psline[linewidth=0.04cm](4.1810155,0.02080078)(4.5810156,-0.3791992)
\psline[linewidth=0.04cm](4.1810155,-0.3791992)(4.5810156,0.02080078)
\end{pspicture} 
}
\\
$n\geq2$&$i=2,\dotsc,n$&
\\
\hline
$C_n\colon{\mathfrak{sp}}(2n,{\mathbb C})$&$\Sigma_i=\{\alpha_i\}$&
\scalebox{0.5} 
{
\begin{pspicture}(0,-0.39919922)(10.80291,0.41919923)
\pscircle[linewidth=0.04,dimen=outer](9.781015,-0.17919922){0.2}
\pscircle[linewidth=0.04,dimen=outer](7.9810157,-0.17919922){0.2}
\pscircle[linewidth=0.04,dimen=outer](6.1810155,-0.17919922){0.2}
\psline[linewidth=0.04cm](6.381016,-0.17919922)(7.7810154,-0.17919922)
\psline[linewidth=0.04cm,linestyle=dotted,dotsep=0.16cm](4.5810156,-0.17919922)(5.9810157,-0.17919922)
\pscircle[linewidth=0.04,dimen=outer](2.5810156,-0.17919922){0.2}
\psline[linewidth=0.04cm,linestyle=dotted,dotsep=0.16cm](2.7810156,-0.17919922)(4.1810155,-0.17919922)
\usefont{T1}{ptm}{m}{n}
\rput(2.5624707,0.22580078){$\alpha_2$}
\usefont{T1}{ptm}{m}{n}
\rput(4.3224707,0.22580078){$\alpha_i$}
\usefont{T1}{ptm}{m}{n}
\rput(6.292471,0.22580078){$\alpha_{n-2}$}
\usefont{T1}{ptm}{m}{n}
\rput(8.092471,0.22580078){$\alpha_{n-1}$}
\usefont{T1}{ptm}{m}{n}
\rput(9.962471,0.22580078){$\alpha_n$}
\psline[linewidth=0.04cm](8.141016,-0.13919921)(9.621016,-0.13919921)
\psline[linewidth=0.04cm](8.141016,-0.21919923)(9.621016,-0.21919923)
\psline[linewidth=0.04cm](8.981015,-0.29919922)(8.741015,-0.17919922)
\psline[linewidth=0.04cm](8.981015,-0.059199218)(8.741015,-0.17919922)
\pscircle[linewidth=0.04,dimen=outer](0.78101563,-0.17919922){0.2}
\psline[linewidth=0.04cm](0.9810156,-0.17919922)(2.3810155,-0.17919922)
\usefont{T1}{ptm}{m}{n}
\rput(0.7624707,0.22580078){$\alpha_1$}
\psline[linewidth=0.04cm](4.1810155,0.02080078)(4.5810156,-0.3791992)
\psline[linewidth=0.04cm](4.1810155,-0.3791992)(4.5810156,0.02080078)
\end{pspicture} 
}
\\
$n\geq3$&$i=1,\dotsc,n-1$&
\\
\hline
$D_n\colon{\mathfrak{so}}(2n,{\mathbb C})$&$\Sigma_i=\{\alpha_i\}$&
\scalebox{0.5} 
{
\begin{pspicture}(0,-1.1080469)(11.06291,1.1080469)
\pscircle[linewidth=0.04,dimen=outer](7.9810157,-0.09035156){0.2}
\pscircle[linewidth=0.04,dimen=outer](6.1810155,-0.09035156){0.2}
\psline[linewidth=0.04cm](6.381016,-0.09035156)(7.7810154,-0.09035156)
\psline[linewidth=0.04cm,linestyle=dotted,dotsep=0.16cm](4.5810156,-0.09035156)(5.9810157,-0.09035156)
\pscircle[linewidth=0.04,dimen=outer](2.5810156,-0.09035156){0.2}
\psline[linewidth=0.04cm,linestyle=dotted,dotsep=0.16cm](2.7810156,-0.09035156)(4.1810155,-0.09035156)
\usefont{T1}{ptm}{m}{n}
\rput(2.5624707,0.31464845){$\alpha_2$}
\usefont{T1}{ptm}{m}{n}
\rput(4.3224707,0.31464845){$\alpha_i$}
\usefont{T1}{ptm}{m}{n}
\rput(6.292471,0.31464845){$\alpha_{n-3}$}
\usefont{T1}{ptm}{m}{n}
\rput(8.092471,0.31464845){$\alpha_{n-2}$}
\usefont{T1}{ptm}{m}{n}
\rput(9.89247,0.9146484){$\alpha_{n-1}$}
\psline[linewidth=0.04cm](8.141016,-0.050351564)(9.581016,0.50964844)
\psline[linewidth=0.04cm](8.141016,-0.13035156)(9.581016,-0.49035156)
\usefont{T1}{ptm}{m}{n}
\rput(9.76247,-0.88535154){$\alpha_n$}
\pscircle[linewidth=0.04,dimen=outer](9.781015,-0.49035156){0.2}
\pscircle[linewidth=0.04,dimen=outer](9.781015,0.50964844){0.2}
\pscircle[linewidth=0.04,dimen=outer](0.78101563,-0.09035156){0.2}
\psline[linewidth=0.04cm](0.9810156,-0.09035156)(2.3810155,-0.09035156)
\psline[linewidth=0.04cm](4.1810155,-0.29035157)(4.5810156,0.10964844)
\psline[linewidth=0.04cm](4.1810155,0.10964844)(4.5810156,-0.29035157)
\usefont{T1}{ptm}{m}{n}
\rput(0.7624707,0.31464845){$\alpha_1$}
\end{pspicture} 
}
\\
$n\geq4$&$i=2,\dotsc,n-2$&
\\
\cline{2-3}
&$\Sigma_{1,n}=\{\alpha_1,\alpha_n\}$&
\scalebox{0.5} 
{
\begin{pspicture}(0,-1.1080469)(9.26291,1.1080469)
\pscircle[linewidth=0.04,dimen=outer](7.9810157,0.50964844){0.2}
\pscircle[linewidth=0.04,dimen=outer](6.1810155,-0.09035156){0.2}
\pscircle[linewidth=0.04,dimen=outer](4.381016,-0.09035156){0.2}
\psline[linewidth=0.04cm](4.5810156,-0.09035156)(5.9810157,-0.09035156)
\pscircle[linewidth=0.04,dimen=outer](2.5810156,-0.09035156){0.2}
\psline[linewidth=0.04cm,linestyle=dotted,dotsep=0.16cm](2.7810156,-0.09035156)(4.1810155,-0.09035156)
\psline[linewidth=0.04cm](0.9810156,-0.09035156)(2.3810155,-0.09035156)
\usefont{T1}{ptm}{m}{n}
\rput(0.7624707,0.31464845){$\alpha_1$}
\usefont{T1}{ptm}{m}{n}
\rput(2.5624707,0.31464845){$\alpha_2$}
\usefont{T1}{ptm}{m}{n}
\rput(4.4924707,0.31464845){$\alpha_{n-3}$}
\usefont{T1}{ptm}{m}{n}
\rput(6.292471,0.31464845){$\alpha_{n-2}$}
\usefont{T1}{ptm}{m}{n}
\rput(8.092471,0.9146484){$\alpha_{n-1}$}
\psline[linewidth=0.04cm](6.341016,-0.050351564)(7.7810154,0.50964844)
\psline[linewidth=0.04cm](6.341016,-0.13035156)(7.7810154,-0.49035156)
\usefont{T1}{ptm}{m}{n}
\rput(7.9624705,-0.88535154){$\alpha_n$}
\psline[linewidth=0.04cm](0.58101565,0.10964844)(0.9810156,-0.29035157)
\psline[linewidth=0.04cm](0.58101565,-0.29035157)(0.9810156,0.10964844)
\psline[linewidth=0.04cm](7.7810154,-0.29035157)(8.181016,-0.69035155)
\psline[linewidth=0.04cm](7.7810154,-0.69035155)(8.181016,-0.29035157)
\end{pspicture} 
}
\\
\cline{2-3}
&$\Sigma_{n-1,n}=\{\alpha_{n-1},\alpha_n\}$&
\scalebox{0.5} 
{
\begin{pspicture}(0,-1.1080469)(9.26291,1.1080469)
\pscircle[linewidth=0.04,dimen=outer](6.1810155,-0.09035156){0.2}
\pscircle[linewidth=0.04,dimen=outer](4.381016,-0.09035156){0.2}
\psline[linewidth=0.04cm](4.5810156,-0.09035156)(5.9810157,-0.09035156)
\pscircle[linewidth=0.04,dimen=outer](2.5810156,-0.09035156){0.2}
\psline[linewidth=0.04cm,linestyle=dotted,dotsep=0.16cm](2.7810156,-0.09035156)(4.1810155,-0.09035156)
\pscircle[linewidth=0.04,dimen=outer](0.78101563,-0.09035156){0.2}
\psline[linewidth=0.04cm](0.9810156,-0.09035156)(2.3810155,-0.09035156)
\usefont{T1}{ptm}{m}{n}
\rput(0.7624707,0.31464845){$\alpha_1$}
\usefont{T1}{ptm}{m}{n}
\rput(2.5624707,0.31464845){$\alpha_2$}
\usefont{T1}{ptm}{m}{n}
\rput(4.4924707,0.31464845){$\alpha_{n-3}$}
\usefont{T1}{ptm}{m}{n}
\rput(6.292471,0.31464845){$\alpha_{n-2}$}
\usefont{T1}{ptm}{m}{n}
\rput(8.092471,0.9146484){$\alpha_{n-1}$}
\psline[linewidth=0.04cm](6.341016,-0.050351564)(7.7810154,0.50964844)
\psline[linewidth=0.04cm](6.341016,-0.13035156)(7.7810154,-0.49035156)
\usefont{T1}{ptm}{m}{n}
\rput(7.9624705,-0.88535154){$\alpha_n$}
\psline[linewidth=0.04cm](7.7810154,0.70964843)(8.181016,0.30964842)
\psline[linewidth=0.04cm](7.7810154,0.30964842)(8.181016,0.70964843)
\psline[linewidth=0.04cm](7.7810154,-0.29035157)(8.181016,-0.69035155)
\psline[linewidth=0.04cm](7.7810154,-0.69035155)(8.181016,-0.29035157)
\end{pspicture} 
}
\\
\hline
\end{tabular}
\end{table}
}}
An important remark is that since automorphisms of a Dynkin diagram induce outer automorphisms of the corresponding Lie algebra, some choices of roots yield isomorphic gradings. An example of this fact are the choices of the roots $\{\alpha_1,\alpha_{n-1}\}$ and $\{\alpha_1,\alpha_n\}$ for the Lie algebra ${\mathfrak{so}}(2n,{\mathbb C})$. In Table~\ref{t:dynsig}, only non-isomorphic graded Lie algebras appear.

Note that $A_1$ does not appear in Table~\ref{t:dynsig}, because it has only one root of height one, and thus it cannot produce a $|2|$-grading. Also we have the following exceptional isomorphisms:
\begin{itemize}
\item ${\mathfrak{so}}(3,{\mathbb C})\cong{\mathfrak{sl}}(2,{\mathbb C})$,

\item ${\mathfrak{sp}}(2,{\mathbb C})\cong{\mathfrak{sl}}(2,{\mathbb C})$,

\item ${\mathfrak{so}}(2,{\mathbb C})\cong{\mathbb C}$,

\item ${\mathfrak{sp}}(4,{\mathbb C})\cong{\mathfrak{so}}(5,{\mathbb C})$,

\item ${\mathfrak{so}}(6,{\mathbb C})\cong{\mathfrak{sl}}(4,{\mathbb C})$,

\item ${\mathfrak{so}}(4,{\mathbb C})\cong{\mathfrak{sl}}(2,{\mathbb C})\times{\mathfrak{sl}}(2,{\mathbb C})$,
\end{itemize}
which explains the choices of $n$ in each case of Table~\ref{t:dynsig}.

\section{Complexified pseudo $H$-type algebras in $|2|$-gradings of classical Lie algebras}\label{sec:class}

We employ Table~\ref{t:dynsig} in order to calculate the dimensions of $\g_{-1}$ and $\g_{-2}$ of a $|2|$-grading of a classical complex Lie algebra $\g$, summarizing the computations in the following statement.

\begin{theorem}\label{th:dimclas}
Let $\g$ be a complex simple Lie algebra, and let $\g_{-2}\oplus\cdots\oplus\g_2$ be a $|2|$-grading of $\g$. Then the dimensions of $\g_{-1}$ and $\g_{-2}$ are given in Table~\ref{t:dim}.
{\rm\small{
\begin{table}[h]
\caption{Choice of $\Sigma$ and dimensions for $|2|$-gradings of the classical Lie algebras}\label{t:dim}
\begin{tabular}{|c|c|c|c|}
\hline
Algebra&$\Sigma$&$\dim(\g_{-1})$&$\dim(\g_{-2})$\\
\hline
${\mathfrak{sl}}(n+1,{\mathbb C})$&$\Sigma_{i,j}=\{\alpha_i,\alpha_j\}$&&
\\
$n\geq2$&$1\leq i\leq\left[\frac{n}2\right]$&$(n+1-(j-i))(j-i)$&$i(n+1-j)$
\\
&$i< j\leq n+1-i$&&
\\
\hline
${\mathfrak{so}}(2n+1,{\mathbb C})$&$\Sigma_i=\{\alpha_i\}$&$i(2(n-i)+1)$&$i(i-1)/2$
\\
$n\geq2$&$i=2,\dotsc,n$&&
\\
\hline
${\mathfrak{sp}}(2n,{\mathbb C})$&$\Sigma_i=\{\alpha_i\}$&$2i(n-i)$&$i(i+1)/2$
\\
$n\geq3$&$i=1,\dotsc,n-1$&&
\\
\hline
${\mathfrak{so}}(2n,{\mathbb C})$&$\Sigma_i=\{\alpha_i\}$&$2i(n-i)$&$i(i-1)/2$
\\
$n\geq4$&$i=2,\dotsc,n-2$&&
\\
\cline{2-4}
&$\Sigma_{1,n}=\{\alpha_1,\alpha_n\}$&$n(n-1)/2$&$n-1$
\\
&&&
\\
\cline{2-4}
&$\Sigma_{n-1,n}=\{\alpha_{n-1},\alpha_n\}$&$2(n-1)$&$(n-1)(n-2)/2$
\\
&&&
\\
\hline
\end{tabular}
\end{table}
}}
\end{theorem}

\begin{proof}
The values given in Table~\ref{t:dim} can be directly found after careful analysis of the matrix representations of the corresponding Lie algebras. For the sake of clarity, let us briefly explain how to obtain diagramatically the dimensions given for the case of ${\mathfrak{sl}}(n+1,{\mathbb C})$. Considering a choice of roots $\Sigma_{i,j}$ according to Table~\ref{t:dynsig}, the induced $|2|$-grading is given schematically by~\eqref{eq:sln}.
Counting dimensions, it is easy to see that
\[
\dim\g_{-1}=(n+1-(j-i))(j-i)\quad\mbox{and}\quad\dim\g_{-2}=i(n+1-j).
\]
For the remaining three classical Lie algebras, the gradings are more involved, but using their matrix representations the computations can be made similarly. 

\[{\tiny
\begin{array}{cccccccccccc}
&&&i&&&&&&&j&
\end{array}}\]
\begin{equation}\label{eq:sln}
{\tiny\begin{array}{c}
\\
\\
i\\
\\
\\
\\
\\
j\\
\\
\\
\\
\end{array}
\left(\begin{array}{ccc|ccccc|ccc}
&&&&&&&&&&\\
&{\mathfrak{g}}_{0}&&&&{\mathfrak{g}}_{1}&&&&{\mathfrak{g}}_{2}&\\
&&&&&&&&&&\\
\hline &&&&&&&&&&\\
&&&&&&&&&&\\
&{\mathfrak{g}}_{-1}&&&&{\mathfrak{g}}_{0}&&&&{\mathfrak{g}}_{1}&\\
&&&&&&&&&&\\
&&&&&&&&&&\\
\hline &&&&&&&&&&\\
&{\mathfrak{g}}_{-2}&&&&{\mathfrak{g}}_{-1}&&&&{\mathfrak{g}}_{0}&\\
&&&&&&&&&&
\end{array}\right).}
\end{equation}
\end{proof}

As a consequence of Theorem~\ref{th:dimclas}, we recover the 
following theorem known by Boothby~\cite{Booth}, see also~\cite{Y}.  

\begin{theorem}\label{th:yamaguchiclass}
For any complex simple Lie algebra $\g$ there exists a unique contact grading, up to isomorphism, that is a $|2|$-grading $\g=\g_{-2}\oplus\cdots\oplus\g_2$ for which $\g_{-2}\oplus\g_{-1}$ is 
the complexified Heisenberg algebra.
\end{theorem}

The list of the contact gradings is given in Table~\ref{t:contactclass}
for the classical simple Lie algebras and later in Section~\ref{sec:excep}  for the exceptional  Lie algebras.
\begin{table}[h]
\caption{Contact cases in complex simple Lie algebras}\label{t:contactclass}
\begin{tabular}{|c|c|c|}
\hline
Type of $\g$&$\dim\g_{-1}$&Choice of root\\
\hline
$A_n$&$2n-2$&$\{\alpha_1,\alpha_n\}$\\
\hline
$B_n$&$4n-6$&$\{\alpha_2\}$\\
\hline
$C_n$&$2n-2$&$\{\alpha_1\}$\\
\hline
$D_n$&$4n-8$&$\{\alpha_2\}$\\
\hline
\end{tabular}
\end{table}

Given a pseudo $H$-type Lie algebra $\n=\n_{-2}\oplus\n_{-1}$, where $\n_{-1}$ is an admissible module of the Clifford algebra ${\rm Cl}(\n_{-2},\langle.\,,.\rangle_{\n_{-2}})$, there is a relation between the dimensions of $\n_{-1}$ and $\n_{-2}$ given in terms of the Radon-Hurwitz-Eckmann function $\rho$. Explicitly, if $n=u2^{4\alpha+\beta}$, where $u$ is odd and $\beta=0,1,2$ or $3$, then $\rho(n):=8\alpha+2^\beta$. For the pseudo $H$-type Lie algebra $\n$, we have that
\begin{equation}\label{eq:RHE}
\dim\n_{-2}\leq \rho(\dim\n_{-1})-1.
\end{equation}
In the case of a positive definite scalar product $\langle.\,,.\rangle_{\n_{-2}}$, it is known that equation~\eqref{eq:RHE} is a necessary and sufficient condition for the existence of an $H$-type Lie algebra of dimension $\dim\n_{-1}+\dim\n_{-2}$ with the center $\n_{-2}$, see~\cite{Kaplan}. For an indefinite non-degenerate scalar product $\langle.\,,.\rangle_{\n_{-2}}$ inequality~\eqref{eq:RHE} is only a necessary condition, see~\cite{Ciatti}. The reason behind this is that not all Clifford modules are admissible, as discussed in Section~\ref{ssec:pseudoH}.

From the definition of the Radon-Hurwitz-Eckmann function, it is clear that if $u$ is odd, then we have
 \[
 \rho(u2^r)=\rho(2^r)=
 \begin{cases}
 8s+1=2r+1, & \text{if $r=4s$},\\
  8s+2=2r, & \text{if $r=4s+1$},\\
   8s+4=2r, & \text{if $r=4s+2$},\\
    8s+8=2r+2, & \text{if $r=4s+3$}.
 \end{cases}
 \]
Based on this relation, we want to find all the nilpotent parts $\g_{-2}\oplus\g_{-1}$ of $|2|$-gradings of the classical Lie algebras which satisfy inequality~\eqref{eq:RHE} in this context, namely $\dim\g_{-2}\leq \rho(\dim\g_{-1})-1$. They are the first candidates for a positive answer to Problem~2.
From now on, we denote $d_1=\dim\g_{-1}$ and $d_2=\dim\g_{-2}$.

\subsection{The case of $A_n$}

Assume $d_2>1$. Since the dimension of $\g_{-1}$ must be divisible by $4$ we exclude the following simple algebras from further consideration.

\begin{proposition}\label{prop:trivAn}
Let the system of roots $\Sigma_{i,j}$ for ${\mathfrak{sl}}(n+1,{\mathbb C})$ is chosen. If either
\begin{enumerate}
\item both $n$ and $j-i$ are odd, or
\item $n$ is even and $j-i\equiv 2\mod 4$,
\end{enumerate}
then the negative part of the $|2|$-grading of ${\mathfrak{sl}}(n+1,{\mathbb C})$ cannot be of pseudo $H$-type.
\end{proposition}

Given $d_2>1$, we want to find all possible choices of roots $\Sigma_{i,j}$ such that the associated negative part of the $|2|$-grading is a candidate for a pseudo $H$-type Lie algebra. 

\begin{theorem}\label{th:An}
For any $d_2>1$, there is $n\in{\mathbb N}$ and a $|2|$-grading of $\g={\mathfrak{sl}}(n+1,{\mathbb C})$, defined by some $\Sigma_{i,j}$, such that $\g_{-2}\oplus\g_{-1}$ corresponds to the complexification of a pseudo $H$-type algebra with the center of dimension $d_2$.
\end{theorem}

\begin{proof}
Since $d_2=i(n+1-j)$ according to Table~\ref{t:dim}, we see that $i$ must divide $d_2$ and $j$ can be found as
$j=n+1-\frac{d_2}i$. 
By the symmetry presented in Table~\ref{t:dynsig}, if $D(d_2)$ denotes the number of divisors of $d_2$, then only the first $\left[\frac{D(d_2)+1}2\right]$ divisors $i$ of $d_2$ have to be analyzed.
It is enough to choose $n$ in such a way that 
\[
d_1=\left(i+\dfrac{d_2}i\right)\left(n+1-\left(i+\dfrac{d_2}i\right)\right)
\]
is a multiple of the dimension of some minimal admissible module for a Clifford algebra generated by a $d_2$-dimensional vector space.
\end{proof}

\begin{coro}
For any $d_2>1$, the nilpotent part of the grading of ${\mathfrak{sl}}(n+1,{\mathbb C})$ determined by $\Sigma_{1,n-d_2}$ satisfies
\[
\dim\g_{-2}=d_2, \quad\mbox{and}\quad\dim\g_{-1}=(d_2+1)(n-d_2),
\]
which correspond to the dimensions $d_2$ and $d_1$, respectively, of a pseudo $H$-type algebra $\n_{-2}\oplus\n_{-1}$, for an appropriate $n$.

If $d_2$ is a prime number, then this is a unique possibility for a pseudo $H$-type algebra to be the nilpotent part of a grading of ${\mathfrak{sl}}(n+1,{\mathbb C})$.
\end{coro}

We can use Theorem~\ref{th:An} with ease for low dimensions. All the cases for $d_2=2,\dotsc,8$ are shown in Table~\ref{t:dimAn}.

\begin{table}[h]
\caption{Examples for Theorem~\ref{th:An}}\label{t:dimAn}
\begin{tabular}{|c|c|c|c|}
\hline
$d_2$&$(i,j)$&$d_1$&Restrictions\\
\hline
$2$&$(1,n-1)$&$3(n-2)$&$n>2$, $n\equiv 2\mod 4$\\
\hline
$3$&$(1,n-2)$&$4(n-3)$&$n>3$\\
\hline
$4$&$(1,n-3)$&$5(n-4)$&$n>4$, $n\equiv 4\mod 8$\\
\hline
&$(2,n-1)$&$4(n-3)$&$n>3$, $n$ odd\\
\hline
$5$&$(1,n-4)$&$6(n-5)$&$n>5$, $n\equiv 1\mod 4$\\
\hline
$6$&$(1,n-5)$&$7(n-6)$&$n>6$, $n\equiv 6\mod 8$\\
\hline
&$(2,n-2)$&$5(n-4)$&$n>4$, $n\equiv 4\mod 8$\\
\hline
$7$&$(1,n-6)$&$8(n-7)$&$n>8$\\
\hline
$8$&$(1,n-8)$&$9(n-8)$&$n>9$, $n\equiv 0\mod 8$\\
\hline
&$(2,n-3)$&$6(n-5)$&$n>5$, $n\equiv 1\mod 4$\\
\hline
\end{tabular}
\end{table}

\subsection{The case of $B_n$}

%

This is the simplest case to analyze, since the only chances of matching the dimensions of the negative part of any $|2|$-grading of ${\mathfrak{so}}(2n+1,{\mathbb C})$ with a pseudo $H$-type Lie algebra is the Heisenberg algebra.

\begin{theorem}\label{th:HtypeBn}
For $d_2>1$ there are no $|2|$-gradings of ${\mathfrak{so}}(2n+1,{\mathbb C})$ such that the dimensions of the negative part match the dimensions of pseudo $H$-type Lie algebras.
\end{theorem}

\begin{proof}
Table~\ref{t:dim} states that $d_1=i(2(n-i)+1)$. The number $(2(n-i)+1)$ is odd, therefore $i$ must be divisible by 4. We write $i=u2^r$, where $r\geq2$ and $u$ is odd. Then, for $r\geq2$, we have
 \[
d_2=\frac{i(i-1)}2=u2^{r-1}(u2^r-1)\geq 2^{r-1}(2^r-1)\geq 2r+2\geq \rho(2^r)=\rho(d_1),
 \]
where we write $d_1=u(2(n-i)+1)2^r$. Since $d_1$ and $d_2$ must satisfy the necessary condition~\eqref{eq:RHE}, we obtain a contradiction.
\end{proof}

\subsection{The case of $C_n$}

In Table~\ref{t:Cn} one can find the possible candidates for a pseudo $H$-type Lie algebra which corresponds to the nilpotent part of a $|2|$-grading of ${\mathfrak{sp}}(2n,{\mathbb C})$ with $d_2>1$. Some explanations are needed in order to understand how Table~\ref{t:Cn} was obtained. The numbering in the second row and in the first column follow because all dimensions $d_2$ are triangular numbers, and the hypothesis $d_2>1$ implies that $d_1$ must be a multiple of four. The numbers that appear in Table~\ref{t:Cn} are solutions for $n$ to the dimension equations
\[
d_1=2i(n-i),\quad d_2=\dfrac{i(i+1)}2.
\]
The numbers in boldface are those values of $n$ for which, in addition, the Radon-Hurwitz-Eckmann condition of equation~\eqref{eq:RHE} holds. Once a number appears in boldface in a column, say in dimension $d_1=d$, the corresponding $n$ also appears in boldface in each $d_1$ multiple of $d$. Once a number appears in boldface in a row, all numbers to the left also appear in boldface.

\begin{table}[h]
\caption{Possible pseudo $H$-type algebras in the grading of ${\mathfrak{sp}}(2n,{\mathbb C})$}\label{t:Cn}
\begin{tabular}{|c||c|c|c|c|c|c|c|c|c|}
\hline
$i$&2&3&4&5&6&7&8&9&10\\
\hline
$d_{-1}\setminus d_{-2}$&3&6&10&15&21&28&36&45&55\\
\hline
\hline
4&{\bf{3}}&--&--&--&--&--&--&--&--\\
\hline
8&{\bf{4}}&--&--&--&--&--&--&--&--\\
\hline
12&{\bf{5}}&5&--&--&--&--&--&--&--\\
\hline
16&{\bf{6}}&--&6&--&--&--&--&--&--\\
\hline
20&{\bf{7}}&--&--&7&--&--&--&--&--\\
\hline
24&{\bf{8}}&{\bf{7}}&7&--&8&--&--&--&--\\
\hline
28&{\bf{9}}&--&--&--&--&--&--&--&--\\
\hline
32&{\bf{10}}&--&8&--&--&--&--&--&--\\
\hline
36&{\bf{11}}&9&--&--&9&--&--&--&--\\
\hline
40&{\bf{12}}&--&9&9&--&--&--&--&--\\
\hline
$\vdots$&$\mathbf{\vdots}$&&$\vdots$&$\vdots$&$\vdots$&$\vdots$&$\vdots$&$\vdots$&$\vdots$\\
\hline
64&{\bf{18}}&--&{\bf{12}}&--&--&--&12&--&--\\
\hline
$\vdots$&{$\vdots$}&&&$\vdots$&$\vdots$&$\vdots$&$\vdots$&$\vdots$&$\vdots$\\
\hline
640&{\bf{162}}&--&{\bf{84}}&{\bf{69}}&--&--&48&--&42\\
\hline
\end{tabular}
\end{table}

We can say more for $|2|$-gradings determined by certain special choices of roots $\Sigma_i$ of the algebra ${\mathfrak{sp}}(2n,{\mathbb C})$. 

\begin{theorem}\label{th:Cn}
For the $|2|$-grading of ${\mathfrak{sp}}(2n,{\mathbb C})$ determined by the choice of root $\Sigma_i$, the following holds:
\begin{enumerate}
\item If $n-i$ is an odd number, then only pseudo $H$-type algebras with $\dim\n_{-2}=3$ can match the dimensions of the grading where $d_1$ is divisible by 4 and this occurs for $i=2$.
\item If $n-i$ is an even number, then there are three possibilities:
\begin{itemize}
\item $n=v2^r+1$, then $i=u2^r+1$, $r\geq1$;
\item $n=v2^r$, then $i=u2^r$, $r>1$;
\item $n=2v$, then $i=2$,
\end{itemize}
where $u$ and $v$ are odd numbers.
\end{enumerate}
\end{theorem}

\begin{proof}
Suppose $n-i$ is odd and $i\geq2$. Then $i$ can be written as $i=u2^r$, where $u$ is odd. Since $d_1=2i(n-i)$ must be divisible by four, we deduce that $r\geq1$. It follows that $d_1=v2^{r+1}$, where $v=(n-i)u$ is an odd number. Therefore, using the Radon-Hurwitz-Eckmann function, we obtain
\[
\rho(d_1)=\rho(2^{r+1})\leq2(r+1)+2=2r+4.
\]
On the other hand,
\[
d_2=\frac{i(i+1)}2=u2^{r-1}(u2^r+1)\geq2^{r+1}(2^r+1).
\]
It follows that $d_2>(2r+4)-1$, for $r\geq2$, contradicting equation~\eqref{eq:RHE}.

If $r=1$, then $d_2=u(2u+1)$, but $\rho(d_1)=\rho(4v)=\rho(4)=4$. Thus $d_2>3$ for all $u>1$, which again contradicts equation~\eqref{eq:RHE}.

The remaining case corresponds to $i=2$, $d_2=3$, and $d_1$ divisible by~4.

Suppose now that  $n-i$ is even. Then the proof for the rest of the cases is the following.
\begin{itemize}
\item Let $n=v2^s+1$ and $i=u2^r+1$, where $1\leq r<s$ are integers and $u$ and $v$ are odd. We have that
$
d_1=2^{r+1}(u2^r+1)(v2^{s-r}-u)$,
and therefore, $\rho(d_1)=\rho(2^{r+1})\leq2(r+1)+2$. On the other hand,
\begin{align*}
d_2=(u2^r+1)(u2^{r-1}+1)&\geq(2^r+1)(2^{r-1}+1)\\&>2(r+1)-1\geq\rho(d_1)-1,
\end{align*}
which contradicts equation~\eqref{eq:RHE}.

\item Let $n=v2^s$ and $i=u2^r$ for integers $r,s$, $1\leq r<s$ and odd $u$ and $v$. We have that $d_1=2^{2r+1}u(v2^{s-r}-u)$, and therefore $\rho(d_1)=\rho(2^{2r+1})\leq2(2r+1)+2$. On the other hand,
\[
d_2=u2^{r-1}(u2^{r}+1)\geq2^{r-1}(2^r+1)>4r+3\geq\rho(d_1)-1,
\]
for all $r>2$, which contradicts equation~\eqref{eq:RHE}. For $r=2$ we have that $\rho(d_1)=\rho(32)=10$, but $d_2\geq10>\rho(d_1)-1=9$, which contradicts equation~\eqref{eq:RHE}.

\item In the case $r=1$, we have that $i=2u$ and $d_1=8u(v2^{s-1}-u)$, thus, $\rho(d_1)=8$ and $d_2=u(2u+1)\leq\rho(d_1)-1=7$ only for $u=1$.
\end{itemize}

If $s<r$, then all previous inequalities become even stronger.
\end{proof}

\subsection{The case of $D_n$}

\begin{theorem}\label{th:HtypeDn1}
Under the assumption $d_2>1$, the only possible pseudo $H$-type algebra occurs for $\Sigma_{4,5}$ among $|2|$-gradings of ${\mathfrak{so}}(2n,{\mathbb C})$ defined by the choice of roots
$\Sigma_{n-1,n}$ and $\Sigma_{1,n}$.
\end{theorem}

\begin{proof}
In the case $\Sigma_{n-1,n}$, the dimensions are $d_1=2(n-1)$ and $d_2=\frac{(n-1)(n-2)}2$, following Table~\ref{t:dim}. Let $2(n-1)=u2^r$, where $u$ is odd. Let us study four different cases, according to whether $r>3$ or $r=1,2,3$.
\begin{itemize}
\item Assuming that $r>3$ we have that
\[
d_2\geq 2^{r-2}(2^{r-1}-1)\geq 2r+2\geq\rho(2^r)=\rho(d_1).
\]
\item If $r=1$, then $d_1=2(n-1)=2u$ implies that $n-1=u\geq3$. Thus, $d_2=\frac12u(u-1)\geq 2=\rho(2u)=\rho(d_1)$.

\item If $r=2$, then $d_1=4u$ and $n-1=2u\geq3$. Then
\[
d_2=u(2u-1)\geq 4=\rho(4u)=\rho(d_1).
\]
\item If $r=3$, then $d_1=8u$ and $n-1=4u\geq3$. Thus, for $u\geq3$, we have that $d_2=2u(4u-1)\geq 8=\rho(8u)=\rho(d_1)$.
\end{itemize}

In all the cases above, we obtained a contradiction with the necessary condition~\eqref{eq:RHE}. It remains to consider the case $r=3$, $u=1$; therefore, $d_1=8$, which implies $n=5$. This is our exceptional case.

In the case $\Sigma_{1,n}$, the dimensions are $d_1=\frac{n(n-1)}2$ and $d_2=n-1$. Let $d_1=u2^r$, for $r\geq2$. Consider two possibilities
\begin{itemize}
\item If $n$ is even, then $n=v2^{r+1}$, with $v$ odd. Thus,
\[
d_2=v2^{r+1}-1\geq2r+2\geq\rho(2^r)=\rho(d_1).
\]

\item If $n-1$ is even, then $n-1=v2^{r+1}$, with $v$ odd. Then
\[
d_2=v2^{r+1}\geq2r+2\geq\rho(2^r)=\rho(d_1).
\]
\end{itemize}
So we obtain again a contradiction with the necessary condition~\eqref{eq:RHE}.
\end{proof}

If we choose the roots $\Sigma_i=\{\alpha_i\}$, $i=2,\dotsc,n-2$, then the table of allowed dimensions for the negative nilpotent part of $|2|$-gradings of ${\mathfrak{so}}(2n,{\mathbb C})$ with respect to the necessary condition~\eqref{eq:RHE} is the same as Table~\ref{t:Cn}, with the shift of 1 to the left in the row of $i$. This is easy to see comparing the $C_n$ case with $D_n$ in Table~\ref{t:dim}. Thus, Theorem~\ref{th:Cn} holds true for ${\mathfrak{so}}(2n,{\mathbb C})$.

\section{Complexified pseudo $H$-type algebras in $|2|$-gradings of the exceptional Lie algebras}\label{sec:excep}

We start by constructing $|2|$-gradings for the exceptional Lie algebras, using the conventions from~\cite{B}. The definitions presented in Section~\ref{ssec:not} show that the only possibilities for $|2|$-gradings of the five exceptional Lie algebras are given in Table~\ref{t:dynsigex}. To obtain the dimension of $\g_{-k}$, $k=1,2$,  of the $|2|$-grading 
it suffices to count the number of roots whose root spaces generate
$\g_{-k}$, and this is done with the tables of roots for exceptional Lie algebras
in~\cite{B}. The results are summarized in Table~\ref{t:dimex}, see also partial results in~\cite{Som}.
Note that each exceptional Lie algebra contains a unique contact grading, as it is stated in Theorem~\ref{th:yamaguchiclass}.  

{\small{
\begin{table}[h]
\caption{Dynkin diagrams and choices of $\Sigma$ for $|2|$-gradings of the exceptional Lie algebras}\label{t:dynsigex}
\begin{tabular}{|c|c|c|}
\hline
Algebra&$\Sigma$&Dynkin diagram\\
\hline
$E_6$&$\Sigma_{2}=\{\alpha_2\}$&
\scalebox{0.5} 
{
\begin{pspicture}(0,-1.5080469)(8.80291,1.5080469)
\usefont{T1}{ptm}{m}{n}
\rput(0.7624707,1.3146484){$\alpha_1$}
\pscircle[linewidth=0.04,dimen=outer](0.78101563,0.9096484){0.2}
\psline[linewidth=0.04cm](0.9810156,0.9096484)(2.3810155,0.9096484)
\pscircle[linewidth=0.04,dimen=outer](4.381016,-0.89035153){0.2}
\pscircle[linewidth=0.04,dimen=outer](2.5810156,0.9096484){0.2}
\usefont{T1}{ptm}{m}{n}
\rput(2.5624707,1.3146484){$\alpha_3$}
\usefont{T1}{ptm}{m}{n}
\rput(4.3624706,1.3146484){$\alpha_4$}
\usefont{T1}{ptm}{m}{n}
\rput(6.162471,1.3146484){$\alpha_5$}
\usefont{T1}{ptm}{m}{n}
\rput(7.9624705,1.3146484){$\alpha_6$}
\psline[linewidth=0.04cm](2.7810156,0.9096484)(4.1810155,0.9096484)
\pscircle[linewidth=0.04,dimen=outer](4.381016,0.9096484){0.2}
\psline[linewidth=0.04cm](4.5810156,0.9096484)(5.9810157,0.9096484)
\pscircle[linewidth=0.04,dimen=outer](6.1810155,0.9096484){0.2}
\psline[linewidth=0.04cm](6.381016,0.9096484)(7.7810154,0.9096484)
\pscircle[linewidth=0.04,dimen=outer](7.9810157,0.9096484){0.2}
\psline[linewidth=0.04cm](4.381016,0.70964843)(4.381016,-0.69035155)
\usefont{T1}{ptm}{m}{n}
\rput(4.3624706,-1.2853515){$\alpha_2$}
\end{pspicture} 
}
\\
&$\Sigma_{3}=\{\alpha_3\}$&\\
&$\Sigma_{1,6}=\{\alpha_1,\alpha_6\}$&\\
\hline
$E_7$&$\Sigma_{1}=\{\alpha_1\}$&
\scalebox{0.5} 
{
\begin{pspicture}(0,-1.5080469)(10.60291,1.5080469)
\usefont{T1}{ptm}{m}{n}
\rput(0.7624707,1.3146484){$\alpha_1$}
\pscircle[linewidth=0.04,dimen=outer](0.78101563,0.9096484){0.2}
\psline[linewidth=0.04cm](0.9810156,0.9096484)(2.3810155,0.9096484)
\pscircle[linewidth=0.04,dimen=outer](4.381016,-0.89035153){0.2}
\pscircle[linewidth=0.04,dimen=outer](2.5810156,0.9096484){0.2}
\usefont{T1}{ptm}{m}{n}
\rput(2.5624707,1.3146484){$\alpha_3$}
\usefont{T1}{ptm}{m}{n}
\rput(4.3624706,1.3146484){$\alpha_4$}
\usefont{T1}{ptm}{m}{n}
\rput(6.162471,1.3146484){$\alpha_5$}
\usefont{T1}{ptm}{m}{n}
\rput(7.9624705,1.3146484){$\alpha_6$}
\psline[linewidth=0.04cm](2.7810156,0.9096484)(4.1810155,0.9096484)
\pscircle[linewidth=0.04,dimen=outer](4.381016,0.9096484){0.2}
\psline[linewidth=0.04cm](4.5810156,0.9096484)(5.9810157,0.9096484)
\pscircle[linewidth=0.04,dimen=outer](6.1810155,0.9096484){0.2}
\psline[linewidth=0.04cm](6.381016,0.9096484)(7.7810154,0.9096484)
\pscircle[linewidth=0.04,dimen=outer](7.9810157,0.9096484){0.2}
\psline[linewidth=0.04cm](4.381016,0.70964843)(4.381016,-0.69035155)
\usefont{T1}{ptm}{m}{n}
\rput(4.3624706,-1.2853515){$\alpha_2$}
\psline[linewidth=0.04cm](8.181016,0.9096484)(9.581016,0.9096484)
\pscircle[linewidth=0.04,dimen=outer](9.781015,0.9096484){0.2}
\usefont{T1}{ptm}{m}{n}
\rput(9.76247,1.3146484){$\alpha_7$}
\end{pspicture} 
}
\\
&$\Sigma_{2}=\{\alpha_2\}$&
\\
&$\Sigma_{6}=\{\alpha_6\}$&
\\
\hline
$E_8$&$\Sigma_{1}=\{\alpha_1\}$&
\scalebox{0.5} 
{
\begin{pspicture}(0,-1.5080469)(12.40291,1.5080469)
\usefont{T1}{ptm}{m}{n}
\rput(0.7624707,1.3146484){$\alpha_1$}
\pscircle[linewidth=0.04,dimen=outer](0.78101563,0.9096484){0.2}
\psline[linewidth=0.04cm](0.9810156,0.9096484)(2.3810155,0.9096484)
\pscircle[linewidth=0.04,dimen=outer](4.381016,-0.89035153){0.2}
\pscircle[linewidth=0.04,dimen=outer](2.5810156,0.9096484){0.2}
\usefont{T1}{ptm}{m}{n}
\rput(2.5624707,1.3146484){$\alpha_3$}
\usefont{T1}{ptm}{m}{n}
\rput(4.3624706,1.3146484){$\alpha_4$}
\usefont{T1}{ptm}{m}{n}
\rput(6.162471,1.3146484){$\alpha_5$}
\usefont{T1}{ptm}{m}{n}
\rput(7.9624705,1.3146484){$\alpha_6$}
\psline[linewidth=0.04cm](2.7810156,0.9096484)(4.1810155,0.9096484)
\pscircle[linewidth=0.04,dimen=outer](4.381016,0.9096484){0.2}
\psline[linewidth=0.04cm](4.5810156,0.9096484)(5.9810157,0.9096484)
\pscircle[linewidth=0.04,dimen=outer](6.1810155,0.9096484){0.2}
\psline[linewidth=0.04cm](6.381016,0.9096484)(7.7810154,0.9096484)
\pscircle[linewidth=0.04,dimen=outer](7.9810157,0.9096484){0.2}
\psline[linewidth=0.04cm](4.381016,0.70964843)(4.381016,-0.69035155)
\usefont{T1}{ptm}{m}{n}
\rput(4.3624706,-1.2853515){$\alpha_2$}
\psline[linewidth=0.04cm](8.181016,0.9096484)(9.581016,0.9096484)
\pscircle[linewidth=0.04,dimen=outer](9.781015,0.9096484){0.2}
\usefont{T1}{ptm}{m}{n}
\rput(9.76247,1.3146484){$\alpha_7$}
\psline[linewidth=0.04cm](9.981015,0.9096484)(11.381016,0.9096484)
\pscircle[linewidth=0.04,dimen=outer](11.581016,0.9096484){0.2}
\usefont{T1}{ptm}{m}{n}
\rput(11.56247,1.3146484){$\alpha_8$}
\end{pspicture} 
}
\\
&$\Sigma_{8}=\{\alpha_8\}$&
\\
\hline
$F_4$&$\Sigma_{1}=\{\alpha_1\}$&
\scalebox{0.5} 
{
\begin{pspicture}(0,-0.3791992)(7.00291,0.41919923)
\pscircle[linewidth=0.04,dimen=outer](4.381016,-0.17919922){0.2}
\pscircle[linewidth=0.04,dimen=outer](6.1810155,-0.17919922){0.2}
\psline[linewidth=0.04cm](4.5810156,-0.17919922)(5.9810157,-0.17919922)
\pscircle[linewidth=0.04,dimen=outer](2.5810156,-0.17919922){0.2}
\usefont{T1}{ptm}{m}{n}
\rput(2.5624707,0.22580078){$\alpha_2$}
\usefont{T1}{ptm}{m}{n}
\rput(4.3624706,0.22580078){$\alpha_3$}
\usefont{T1}{ptm}{m}{n}
\rput(6.162471,0.22580078){$\alpha_4$}
\psline[linewidth=0.04cm](2.7410157,-0.13919921)(4.2210155,-0.13919921)
\psline[linewidth=0.04cm](2.7410157,-0.21919923)(4.2210155,-0.21919923)
\psline[linewidth=0.04cm](3.5810156,-0.17919922)(3.3410156,-0.07919922)
\psline[linewidth=0.04cm](3.5810156,-0.15919922)(3.3410156,-0.2791992)
\pscircle[linewidth=0.04,dimen=outer](0.78101563,-0.17919922){0.2}
\psline[linewidth=0.04cm](0.9810156,-0.17919922)(2.3810155,-0.17919922)
\usefont{T1}{ptm}{m}{n}
\rput(0.7624707,0.22580078){$\alpha_1$}
\end{pspicture} 
}
\\
&$\Sigma_{4}=\{\alpha_4\}$&
\\
\hline
$G_2$&$\Sigma_{2}=\{\alpha_2\}$&
\scalebox{0.5} 
{
\begin{pspicture}(0,-0.3791992)(3.4029102,0.41919923)
\pscircle[linewidth=0.04,dimen=outer](2.5810156,-0.17919922){0.2}
\pscircle[linewidth=0.04,dimen=outer](0.78101563,-0.17919922){0.2}
\usefont{T1}{ptm}{m}{n}
\rput(0.7624707,0.22580078){$\alpha_1$}
\usefont{T1}{ptm}{m}{n}
\rput(2.5624707,0.22580078){$\alpha_2$}
\psline[linewidth=0.04cm](0.9410156,-0.09919922)(2.4210157,-0.09919922)
\psline[linewidth=0.04cm](0.9410156,-0.25919923)(2.4210157,-0.25919923)
\psline[linewidth=0.04cm](1.8210156,-0.33919922)(1.5410156,-0.17919922)
\psline[linewidth=0.04cm](1.8210156,0.0)(1.5410156,-0.17919922)
\psline[linewidth=0.04cm](0.9410156,-0.17919922)(2.3810155,-0.17919922)
\end{pspicture} 
}
\\
\hline
\end{tabular}
\end{table}
}}

\begin{table}[h]
\caption{Choice of $\Sigma$ and dimensions for $|2|$-gradings of the exceptional Lie algebras}\label{t:dimex}
\begin{tabular}{|c|c|c|c|}
\hline
Algebra&$\Sigma$&$\dim(\g_{-1})$&$\dim(\g_{-2})$\\
\hline
&$\Sigma_2$&$20$&1\\
$E_6$&$\Sigma_3$&{{20}}&{{5}}\\
&$\Sigma_{1,6}$&${\bf16}$&$\bf{8}$\\
\hline
&$\Sigma_1$&$32$&$1$\\
$E_7$&$\Sigma_2$&$35$&$7$\\
&$\Sigma_6$&$32$&$10$\\
\hline
$E_8$&$\Sigma_1$&$64$&$14$\\
$E_8$&$\Sigma_8$&$56$&$1$\\
\hline
$F_4$&$\Sigma_1$&$14$&$1$\\
$F_4$&$\Sigma_4$&{\bf{8}}&{\bf{7}}\\
\hline
$G_2$&$\Sigma_2$&$4$&1
\\
\hline
\end{tabular}
\end{table}


\begin{theorem}\label{th:exH}
Only $(E_6,  \Sigma _{1, 6} )$ and  $(F_4, \Sigma _4 ) $  are the exceptional simple graded Lie algebras
$\g = \bigoplus \g_p $ whose negative part $\g_-$  is the complexification 
$\n \otimes \mathbb C$ of  a pseudo $H$-type Lie algebra
$ \n =  \n _{-2} + \n_{-1} $ with dim $\n_{-2} > 1$.
\end{theorem}

\begin{proof}
It follows from Table~\ref{t:dimex} and from the Radon-Hurwitz-Eckmann inequality~\eqref{eq:RHE} that 
only $(E_6,  \Sigma _{1, 6} )$ and  $(F_4, \Sigma _4 ) $  are the exceptional simple graded Lie algebras
that satisfy the necessary conditions. In fact, among the graded exceptional Lie algebras
in Table~\ref{t:dimex}, the contact gradings are excluded from the candidates. Computing the 
number $\rho (\dim \g _{-1} ) - 1 - \dim \g_{-2} $ for the other gradings, we note that
only $(E_6,  \Sigma _{1, 6} )$ and  $(F_4, \Sigma _4 ) $ satisfy the  Radon-Hurwitz-Eckmann inequality~\eqref{eq:RHE},
since $\rho(x) = 4, 9, 1, 10, 12, 8$ for $x = 20, 16, 15, 32, 64, 8 $.
 
Now we show by a computer-aided proof that there actually exist pseudo H-type algebras 
satisfying Theorem~\ref{th:exH}. The prolongation $\rm{Prol}( \n^{8, 0}) $ of
$ \n^{8, 0} $ is computed by making use of the Maple program [TanakaProlongation]. It was found that the growth vector of $\rm{Prol}( \n^{8, 0}) $ is
 $ (8, 16, 30, 16, 8) $ and that the prolongation is semi-simple.
Since the minimal  admissible module of $\Cl_{8,0} $ is irreducible, the prolongation $\rm{Prol}( \n^{8, 0}) $ is simple. Due to the fact that a complex simple graded Lie algebra having this growth vector is only $(E_6,  \Sigma _{1, 6} )$,
we conclude that the prolongation $\rm{Prol}( \n^{8,0})$ is a real form of $(E_6,  \Sigma _{1, 6 } )$. 
  
Computing the prolongations of $\n^{7,1} $, $\n^{4,4}$, and $\n^{3, 5}$,
one finds that $\rm{Prol}(\n^{7,1}) $,  $\rm{Prol}(\n^{4,4})$, and $\rm{Prol}(\n^{3,5}) $ are also real forms of $(E_6,  \Sigma_{1, 6} )$.

Again by the Maple calculations one can see that the growth vectors of prolongations 
$\rm{Prol}( \n^{7,0} )$ and  $\rm{Prol}( \n^{3, 4}) $ are $(7, 8, 22, 8, 7) $, from 
which we conclude that  $\rm{Prol}( \n^{7,0}) $ and $\rm{Prol}( \n^{3, 4} )$
are real forms of $(F_4, \Sigma _4 )$. 
\end{proof}


\section{A real example: $\n^{1,3}$ inside ${\mathfrak{su}}(3,3)$}\label{sec:example}


Consider the Lie algebra ${\mathfrak{su}}(3,3)$, which is a real form of the complex simple Lie algebra ${\mathfrak{sl}}(6,{\mathbb C})$, defined as
\[
{\mathfrak{su}}(3,3)=\{X\in{\mathfrak{gl}}(6,{\mathbb C})\colon X+\sigma X^*\sigma=0\},
\]
where $X^*$ denotes the transposed conjugate matrix of $X$ and
\[
\sigma=\left(\begin{array}{cccccc}0&0&0&0&0&1\\0&0&0&0&1&0\\0&0&0&1&0&0\\0&0&1&0&0&0\\0&1&0&0&0&0\\1&0&0&0&0&0\end{array}\right).
\]

The $|2|$-grading of ${\mathfrak{sl}}(6,{\mathbb C})$ induced by the choice of roots $\Sigma_{2,4}$ gives a nilpotent Lie algebra $\g_{-2}\oplus\g_{-1}$ with dimensions $\dim\g_{-1}=8$ and $\dim\g_{-2}=4$.
This grading has its natural counterpart $\n_{-2}\oplus\n_{-1}$ in the grading of the real Lie algbra ${\mathfrak{su}}(3,3)$ satisfying the following statement. 

\begin{theorem}
The graded 2-step nilpotent Lie algebra $\n_{-2}\oplus\n_{-1}$ is isomorphic to the pseudo $H$-type algebra $\n^{1,3}$ with $\n_{-1}$ as the minimal admissible module of ${\rm Cl}(\n_{-2},\langle.\,,.\rangle_{\n_{-2}})$, where $\langle.\,,.\rangle_{\n_{-2}}$ is the scalar product of signature $(1,3)$.
\end{theorem}

\begin{proof}

Let $E_{ij}$ denote the $(6\times6)$-matrix with 1 at the position $(i,j)$ and $0$ elsewhere. The set of matrices
\begin{equation}\label{eq:basisn1}
\begin{array}{ll}
e_1 = E_{41} - E_{63}+E_{32} - E_{54} ,& e_2 = i (E_{41} + E_{63}-E_{32} - E_{54}),\\
e_3 = E_{31} - E_{64} - E_{42} + E_{53} ,& e_4 = i (E_{31} + E_{64}+E_{42} + E_{53}) ,\\
e_5 = E_{41} - E_{63} - E_{32} + E_{54},& e_6 = i (E_{41} + E_{63}+E_{32} + E_{54}),\\
e_7 =  E_{31} - E_{64} + E_{42} - E_{53},& e_8 = i (E_{31} + E_{64}-E_{42} - E_{53}) ,
\end{array}
\end{equation}
is a basis of $\n_{-1}$, and
\begin{equation}\label{eq:basisn2}
\begin{array}{ll}
z_1=2(E_{51}-E_{62}),&z_2=2i(E_{51}+E_{62}),\\
z_3=2i(E_{52}-E_{61}),&z_4=2i(E_{52}+E_{61}),
\end{array}
\end{equation}
is a basis of $\n_{-2}$. 
Consider the scalar product $\langle.\,,.\rangle_{\n_{-1}}$ on $\n_{-1}$ by declaring the basis~\eqref{eq:basisn1} to be orthonormal, where
\[
\langle e_1,e_1\rangle_{\n_{-1}}=\langle e_2,e_2\rangle_{\n_{-1}}=\langle e_3,e_3\rangle_{\n_{-1}}=\langle e_4,e_4\rangle_{\n_{-1}}=1, 
\]
\[
\langle e_5,e_5\rangle_{\n_{-1}}=\langle e_6,e_6\rangle_{\n_{-1}}=\langle e_7,e_7\rangle_{\n_{-1}}=\langle e_8,e_8\rangle_{\n_{-1}}=-1.
\]
Similarly, consider a scalar product $\langle.\,,.\rangle_{\n_{-2}}$ on $\n_{-2}$ by declaring the basis~\eqref{eq:basisn2} to be orthonormal. Namely
\[
-\langle z_1,z_1\rangle_{\n_{-2}}=-\langle z_2,z_2\rangle_{\n_{-2}}=-\langle z_3,z_3\rangle_{\n_{-2}}=\langle z_4,z_4\rangle_{\n_{-2}}=1.
\]
The following is the commutator table of $\n_{-2}\oplus\n_{-1}$ in terms of the above chosen basis $\{e_1,\dotsc,e_8,z_1,\dotsc,z_4\}$, using the usual commutators of matrices:
{\small{
\begin{table}[h]
\begin{tabular}{|c|c|c|c|c|c|c|c|c|}
\hline
$$&$e_1$&$e_2$&$e_3$&$e_4$&$e_5$&$e_6$&$e_7$&$e_8$\\
\hline
$e_1$&0&0&0&$-z_4$&$-z_1$&0&$-z_2$&$z_3$\\
\hline
$e_2$&0&0&$-z_4$&0&0&$z_1$&$z_3$&$z_2$\\
\hline
$e_3$&0&$z_4$&0&0&$-z_2$&$-z_3$&$z_1$&0\\
\hline
$e_4$&$z_4$&0&0&0&$-z_3$&$z_2$&0&$-z_1$\\
\hline
$e_5$&$z_1$&0&$z_2$&$z_3$&0&0&0&$-z_4$\\
\hline
$e_6$&0&$-z_1$&$z_3$&$-z_2$&0&0&$-z_4$&0\\
\hline
$e_7$&$z_2$&$-z_3$&$-z_1$&0&0&$z_4$&0&0\\
\hline
$e_8$&$-z_3$&$-z_2$&0&$z_1$&$z_4$&0&0&0\\
\hline
\end{tabular}
\end{table}
}}

For each $z\in\n_{-2}$, we define the map $J_z\in{\rm End}(\n_{-1})$ by means of identity~\eqref{eq:Jmap}. In terms of the given basis of $\n_{-1}$, we can easily compute the following matrix representation of $J_{z_i}$, $i=1,2,3,4$.
{\small{
\[
J_{z_1}=\left(
\begin{array}{cccccccc}
 0 & 0 & 0 & 0 & 1 & 0 & 0 & 0 \\
 0 & 0 & 0 & 0 & 0 & -1 & 0 & 0 \\
 0 & 0 & 0 & 0 & 0 & 0 & -1 & 0 \\
 0 & 0 & 0 & 0 & 0 & 0 & 0 & 1 \\
 1 & 0 & 0 & 0 & 0 & 0 & 0 & 0 \\
 0 & -1 & 0 & 0 & 0 & 0 & 0 & 0 \\
 0 & 0 & -1 & 0 & 0 & 0 & 0 & 0 \\
 0 & 0 & 0 & 1 & 0 & 0 & 0 & 0 \\
\end{array}
\right)\!\!,\,
J_{z_2}=\left(
\begin{array}{cccccccc}
 0 & 0 & 0 & 0 & 0 & 0 & 1 & 0 \\
 0 & 0 & 0 & 0 & 0 & 0 & 0 & -1 \\
 0 & 0 & 0 & 0 & 1 & 0 & 0 & 0 \\
 0 & 0 & 0 & 0 & 0 & -1 & 0 & 0 \\
 0 & 0 & 1 & 0 & 0 & 0 & 0 & 0 \\
 0 & 0 & 0 & -1 & 0 & 0 & 0 & 0 \\
 1 & 0 & 0 & 0 & 0 & 0 & 0 & 0 \\
 0 & -1 & 0 & 0 & 0 & 0 & 0 & 0 \\
\end{array}
\right)\!,
\]}}
{\small{\[
J_{z_3}=\left(
\begin{array}{cccccccc}
 0 & 0 & 0 & 0 & 0 & 0 & 0 & -1 \\
 0 & 0 & 0 & 0 & 0 & 0 & -1 & 0 \\
 0 & 0 & 0 & 0 & 0 & 1 & 0 & 0 \\
 0 & 0 & 0 & 0 & 1 & 0 & 0 & 0 \\
 0 & 0 & 0 & 1 & 0 & 0 & 0 & 0 \\
 0 & 0 & 1 & 0 & 0 & 0 & 0 & 0 \\
 0 & -1 & 0 & 0 & 0 & 0 & 0 & 0 \\
 -1 & 0 & 0 & 0 & 0 & 0 & 0 & 0 \\
\end{array}
\right)\!\!,\,
J_{z_4}=\left(
\begin{array}{cccccccc}
 0 & 0 & 0 & 1 & 0 & 0 & 0 & 0 \\
 0 & 0 & 1 & 0 & 0 & 0 & 0 & 0 \\
 0 & -1 & 0 & 0 & 0 & 0 & 0 & 0 \\
 -1 & 0 & 0 & 0 & 0 & 0 & 0 & 0 \\
 0 & 0 & 0 & 0 & 0 & 0 & 0 & -1 \\
 0 & 0 & 0 & 0 & 0 & 0 & -1 & 0 \\
 0 & 0 & 0 & 0 & 0 & 1 & 0 & 0 \\
 0 & 0 & 0 & 0 & 1 & 0 & 0 & 0 \\
\end{array}
\right)\!.
\]}}

Simple computations show that
\[
J_{z_1}^2=J_{z_2}^2=J_{z_3}^2=-J_{z_4}^2={\rm Id}_{\n_{-1}}\quad\mbox{and}\quad J_{z_i}J_{z_j}+J_{z_j}J_{z_i}=0\quad\text{for}\quad i\neq j.
\]
This implies that $J_z^2=-\langle z,z\rangle_{\n_{-2}}{\rm Id}_{\n_{-1}}$, for all $z\in\n_{-2}$, and thus, the 2-step nilpotent Lie algebra $\n_{-2}\oplus\n_{-1}$ is the pseudo $H$-type algebra~$\n^{1,3}$.
\end{proof}

\section{Uniqueness of the Heisenberg algebra in ${\mathfrak{sl}}(n+1,{\mathbb R})$}\label{sec:noHAn}

The aim of this section is to prove that among all $|2|$-gradings $\g_{-2}\oplus\cdots\oplus\g_2$ of $\g={\mathfrak{sl}}(n+1,{\mathbb R})$ the only $H$-type algebras with a positive definite scalar product that appear as $\g_{-2}\oplus\g_{-1}$ are the Heisenberg algebras. To prove this, we will use an alternative definition of $H$-type algebras found in~\cite{Kaplan}. A real graded 2-step nilpotent Lie algebra $\n=\n_{-2}\oplus\n_{-1}$ endowed with a positive definite scalar product $\langle.\,,.\rangle$ is of $H$-type if and only if the map
$
{\rm ad}_x\colon\ker({\rm ad}_x)^\bot\to\n_{-2}
$
is a surjective isometry for any $x\in\n_{-1}$ with $\langle x,x\rangle=1$, where $^\bot$ stands for the orthogonal complement in $\n_{-1}$.

Let us assume that the $|2|$-grading is induced by the choice of roots $\Sigma_{i,j}$. With analogous notation as in Section~\ref{sec:example}, define the following subspaces of ${\mathfrak{g}}_{-1}$
\[
{\mathfrak{g}}_{-1}^L={\rm span}\{E_{k\ell}\colon i<k\leq j,1\leq\ell\leq i\},
\]
\[ 
{\mathfrak{g}}_{-1}^R={\rm span}\{E_{k\ell}\colon j<k\leq n+1,i<\ell\leq j\}.
\]
It is easy to see that ${\mathfrak{g}}_{-1}={\mathfrak{g}}_{-1}^L\oplus{\mathfrak{g}}_{-1}^R$, as vector spaces, and that
\[
\dim{\mathfrak{g}}_{-1}^L=i(j-i),\quad\dim{\mathfrak{g}}_{-1}^R=(n+1-j)(j-i).
\]
It follows immediately that $\dim{\mathfrak{g}}_{-1}=(n+1-(j-i))(j-i)$.
We state a result concerning the products and matrix commutators of the elementary matrices $E_{k\ell}$.

\begin{lemma}\label{l:commAn}
The product of two elementary matrices is given by
\begin{equation}\label{eq:prodel}
E_{pq}E_{rs}=\begin{cases}E_{ps},&\mbox{if }q=r;\\0,&\mbox{otherwise.}\end{cases}
\end{equation}

The commutator between elementary matrices is given by
\begin{equation}\label{eq:commel}
[E_{pq},E_{rs}]=E_{pq}E_{rs}-E_{rs}E_{pq}=
\begin{cases}
E_{ps},&\mbox{if }p\neq s,q=r;\\-E_{rq},&\mbox{if }p=s,q\neq r;\\E_{pp}-E_{qq},&\mbox{if }p=s,q=r;\\0,&\mbox{otherwise.}
\end{cases}
\end{equation}
\end{lemma}

\begin{proof}
We prove only equality~\eqref{eq:prodel}, since~\eqref{eq:commel} is an easy consequence. By the definition of the product of matrices, we have that
\[
(E_{pq}E_{rs})_{\alpha\beta}=\sum_{k=1}^n(E_{pq})_{\alpha k}(E_{rs})_{k\beta}.
\]
It is easy to see that the only entry of the matrix $E_{pq}E_{rs}$ that is not zero is $(E_{pq}E_{rs})_{ps}$. It is also clear that $(E_{pq}E_{rs})_{ps}=1$ only when $q=r$, and vanishes otherwise. This proves equality~\eqref{eq:prodel}.
\end{proof}

With Lemma~\ref{l:commAn} in hands, we can conclude that the subspaces ${\mathfrak{g}}_{-1}^L$ and ${\mathfrak{g}}_{-1}^R$ have the following structure.

\begin{coro}
The subspaces ${\mathfrak{g}}_{-1}^L$ and ${\mathfrak{g}}_{-1}^R$, endowed with the usual matrix commutator, form abelian Lie subalgebras of ${\mathfrak{g}}_{-1}$.
\end{coro}

Since the 2-step nilpotent Lie algebra $\g_{-2}\oplus\g_{-1}$ consists of lower triangular matrices, we have the following result.

\begin{lemma}\label{th:HtypeAn}
Let $\langle.\,,.\rangle$ be any positive definite scalar product defined on $\g={\mathfrak{sl}}(n+1,{\mathbb R})$ such that the elementary matrices
$
E_{k\ell}$, $k,\ell\in\{1,\dotsc,n+1\}$,
form an orthonormal set with respect to $\langle.\,,.\rangle$.
\begin{enumerate}
\item If $E_{k\ell}\in{\mathfrak{g}}_{-1}^R$, then $\dim\Big(\ker({\rm ad}_{E_{k\ell}})^\bot\cap{\mathfrak{g}}_{-1}\Big)= i$.

\item If $E_{k\ell}\in{\mathfrak{g}}_{-1}^L$, then $\dim\Big(\ker({\rm ad}_{E_{k\ell}})^\bot\cap{\mathfrak{g}}_{-1}\Big)= n+1-j$.
\end{enumerate}
\end{lemma}

\begin{proof}
Let $E_{k\ell}\in{\mathfrak{g}}_{-1}$. It is clear from~\eqref{eq:commel}, that
$
\ker({\rm ad}_{E_{k\ell}})=\spn\{E_{pq}\colon p\neq\ell\mbox{ and }q\neq k\}$.
From this equality we can deduce that for any positive definite scalar product $\langle.\,,.\rangle$ as in the statement of the lemma, we have
\[
\ker({\rm ad}_{E_{k\ell}})^\bot=\spn\{E_{pq}\colon p=\ell\mbox{ or }q= k\}.
\]
It is easy to see that if $E_{k\ell}\in{\mathfrak{g}}_{-1}^L$, then
\[
\ker({\rm ad}_{E_{k\ell}})^\bot\cap\g_{-1}=\spn\{E_{pk}\colon p=j+1,\dotsc,n+1\},
\]
and if $E_{k\ell}\in{\mathfrak{g}}_{-1}^R$, then
$
\ker({\rm ad}_{E_{k\ell}})^\bot\cap\g_{-1}=\spn\{E_{\ell q}\colon q=1,\dotsc,i\}$.
Counting dimensions, the claim follows.
\end{proof}

\begin{theorem}\label{th:positive}
The only graded subalgebras of ${\mathfrak{sl}}(n+1,{\mathbb R})$, which are $H$-type algebras with a positive definite scalar product are the Heisenberg algebras of dimension $2n-1$.
\end{theorem}

\begin{proof}
With the notations as in Lemma~\ref{th:HtypeAn}, the map
$
{\rm ad}_{E_{k\ell}}\colon\ker({\rm ad}_{E_{k,\ell}})^\bot\cap{\mathfrak{g}}_{-1}\to{\mathfrak{g}}_{-2}
$
is surjective for any $E_{k\ell}\in{\mathfrak{g}}_{-1}$ only when $i=1$ and $j=n-1$. 
\end{proof}


\section{Prolongations of $\n^{r,s}$}\label{exAppendix}

 
We recall some definitions and basic facts on prolongations of graded Lie algebras
following~\cite{Mo88}.  By a graded Lie algebra we mean a Lie algebra $\g$ endowed with a $\mathbb Z$-grading:
$\g = \bigoplus _{ p \in \mathbb Z } \g_p$. We denote the negative part of $\g$ by 
$\g_- = \bigoplus_{ p < 0 } \g_p $. A graded Lie algebra $\g$ is called { \it transitive}  
if $\dim \g_ - < \infty $ and  if it satisfies the following condition:
$$
\text {if }  p \ge 0, \ x \in \g_p,  \text{ and } [x, \g_-] = 0, \text {then } x = 0 .
$$
The largest  integer $\mu$ such that $\g_{-\mu} \neq 0 $ is called the depth of $\g$.
 The vector 
 $$(\dim\g_{-\mu},\dim\g_{-\mu+1},\ldots,\dim\g_{p},\ldots)
 $$ 
 is called the growth vector of $\g$.

Given a transitive graded Lie algebra $\g = \bigoplus_{ p \in \mathbb Z } \g_p $, then for each $k \ge 0 $ 
there exists a unique maximal transitive graded Lie algebra $\tilde{\g} = \bigoplus_{ p \in \mathbb Z } \tilde{\g}_p$ up to isomorphisms
such that $\tilde{\g}_p = \g_p $ for $p < k$, which is called the prolongation (or Tanaka prolongation)
of the truncated graded Lie algebra $\bigoplus _{p < k } \g_p$.
The prolongation $\tilde{\g}$ is determined inductively by 
$$
\tilde{\g} _{p+1}  = \{ z \in  \Hom ( \g_-, \tilde{\g})_{p+1}
\vert\  [z (x),y] + [x,z(y) ] = z ([x, y] ),\  \text{for all}\ x, y \in \g_-  \},
$$
where $\Hom(*\,, * ) _p$ denotes the set of homogeneous degree $p$ linear maps.
The bracket of $\tilde{\g}$ is defined inductively to satisfy the Jacobi identity
and $z(x) = [z, x] $ for $z \in \tilde{\g}_q $,  $ x \in \g_- $.
The notion of prolongation of depth $\mu = 1$,  had been fundamental in geometry 
and in the theory of partial differential equations, see for instance~\cite{GuilSt}.
This  was generalized to the case of depth greater than 1 in~\cite{Ta} 
 and has played important role in nilpotent geometry and analysis~\cite{Mo2}.

Note that a transitive graded Lie algebra $\g $ is the prolongation of a truncated graded algebra $\bigoplus _ { p < k } \g_p $ if and only if the generalized Spencer cohomology group $H^1 _r ( \g_-, \g)$ associated with the representation of $\g_-$  on $\g$ vanishes for $ r \ge k $.
 
In the case of a simple Lie algebra $\g$, the Spencer cohomology group $ H^p _r ( \g_-, \g) $
may be computed by the method of Kostant~\cite{Kost},  and  Yamaguchi~\cite{Y} carried out
the computation for $H^1$, $H^2$, according to which a simple graded Lie algebra 
 is a prolongation of $\g_-$ except for a few cases.

If a transitive graded Lie algebra $\g= \bigoplus _{ p \in \mathbb Z } \g_p$ is simple and finite dimensional, then the Killing form $B$ of $\g$ satisfies $B( \g_p, \g_q ) = 0 $ if $p+q \neq 0 $, and therefore, the grading is symmetric; that is $B$ is a non-degenerate pairing between $\g_p$ and $\g_{-p}$.

The program [TanakaProlongation] in the Maple Software elaborated by Anderson~\cite{AnI} allows
one to calculate the prolongation of a truncated transitive graded Lie algebra 
$\bigoplus _{ p < k } \g_p $.
If one enters into the Maple the data of a truncated graded Lie algebra: a basis and the structure constants, then the [TanakaProlongation] gives back a basis and the structure equations of the prolongation. In particular, we get the dimension of the prolongation at each order. 

We used Maple to calculate the prolongation of  
the basic pseudo $H$-type Lie algebras $\n^{r,s}=\n_{-2}\oplus\n_{-1}=\mathbb R^{r,s}\oplus\n_{-1}$ for $r+s \le 8$, where $\n_{-1}$ is the minimal dimensional 
admissible module of the Clifford algebra $\Cl(\mathbb R^{r,s},\langle.\,,.\rangle_{r,s})$. 
The Lie algebra data of $\n^{r,s} $ are taken from the tables of structure constants
found in~\cite{FM3}. 
We collect the results of computation in Tables~\ref{T8} and~\ref{T9},
where $\g=\rm{Prol(\n)}$ is the prolongation of a pseudo $H$-type Lie algebra $\n$.  
\begin{table}[h]
\caption{Prolongation data for $n^{r,s}$ with $\dim\n_{-2}=1,2,3,4$.}
\begin{tabular}{|c|c|c|c|c|}
\hline
$\dim\n_{-2}$&$\n$&Growth vector of $\g$&$\g$&$\g\otimes{\mathbb C}$\\
\hline
1&$\n^{1,0}$&$(1,2,4,6,9,\dotsc)$&${\rm ct}(3,{\mathbb R})$&${\rm ct}(3,{\mathbb C})$\\
\hline
&$\n^{0,1}$&$(1,2,4,6,9,\dotsc)$&${\rm ct}(3,{\mathbb R})$&${\rm ct}(3,{\mathbb C})$\\
\hline
2&$\n^{2,0}$&$(2,4,8,12,18,\dotsc)$&${\rm ct}(3,{\mathbb C})_{\mathbb R}$&${\rm ct}(3,{\mathbb C})\oplus{\rm ct}(3,{\mathbb C})$\\
\hline
&$\n^{1,1}$&$(2,4,8,12,18,\dotsc)$&${\rm ct}(3,{\mathbb R})\oplus{\rm ct}(3,{\mathbb R})$&${\rm ct}(3,{\mathbb C})\oplus{\rm ct}(3,{\mathbb C})$\\
\hline
&$\n^{0,2}$&$(2,4,8,12,18,\dotsc)$&${\rm ct}(3,{\mathbb C})_{\mathbb R}$&${\rm ct}(3,{\mathbb C})\oplus{\rm ct}(3,{\mathbb C})$\\
\hline
3&$\n^{3,0}$&$(3,4,7,4,3)$&${\mathfrak{sp}}(2,1)$&${\mathfrak{sp}}({6,\mathbb C})$\\
\hline
&$\n^{2,1}$&$(3,8,14,8,3)$&${\mathfrak{sp}}({8,\mathbb R})$&${\mathfrak{sp}}({8,\mathbb C})$\\
\hline
&$\n^{1,2}$&$(3,4,7,4,3)$&${\mathfrak{sp}}({6,\mathbb R})$&${\mathfrak{sp}}({6,\mathbb C})$\\
\hline
&$\n^{0,3}$&$(3,8,14,8,3)$&${\mathfrak{sp}}({2,2})$&${\mathfrak{sp}}({8,\mathbb C})$\\
\hline
4&$\n^{4,0}$&$(4,8,11,8,4)$&${\mathfrak{sl}}(3,{\mathbb H})$&${\mathfrak{sl}}({6,\mathbb C})$\\
\hline
&$\n^{3,1}$&$(4,8,11,8,4)$&${\mathfrak{su}}(4,2)$&${\mathfrak{sl}}({6,\mathbb C})$\\
\hline
&$\n^{2,2}$&$(4,8,11,8,4)$&${\mathfrak{sl}}({6,\mathbb R})$&${\mathfrak{sl}}({6,\mathbb C})$\\
\hline
&$\n^{1,3}$&$(4,8,11,8,4)$&${\mathfrak{su}}(3,3)$&${\mathfrak{sl}}({6,\mathbb C})$\\
\hline
&$\n^{0,4}$&$(4,8,11,8,4)$&${\mathfrak{sl}}(3,{\mathbb H})$&${\mathfrak{sl}}({6,\mathbb C})$\\
\hline
\end{tabular}\label{T8}
\end{table}

\begin{table}[h]
\caption{Growth vectors for the prolongation of some pseudo $H$-type algebras with $\dim\n_{-2}=5,6,7,8$.}
\begin{tabular}{|c|c|c|c|c|}
\hline
$\dim\n_{-2}$&$\n$&Growth vector of $\g$&$\g$&$\g\otimes{\mathbb C}$\\
\hline
5&$\n^{5,0}$&$(5,8,12)$&&\\
\hline
&$\n^{4,1}$&$(5,16,17)$&&\\
\hline
&$\n^{3,2}$&$(5,8,12)$&&\\
\hline
&$\n^{2,3}$&$(5,8,12)$&&\\
\hline
&$\n^{1,4}$&$(5,8,12)$&&\\
\hline
&$\n^{0,5}$&$(5,16,17)$&&\\
\hline
6&$\n^{6,0}$&$(6,8,16)$&&\\
\hline
&$\n^{5,1}$&$(6,16,18)$&&\\
\hline
&$\n^{4,2}$&$(6,16,18)$&&\\
\hline
&$\n^{3,3}$&$(6,8,16)$&&\\
\hline
&$\n^{2,4}$&$(6,8,16)$&&\\
\hline
&$\n^{1,5}$&$(6,16,18)$&&\\
\hline
&$\n^{0,6}$&$(6,16,18)$&&\\
\hline
7&$\n^{7,0}$&$(7,8,22,8,7)$&$F_{\rm{II}}$&$F_4$\\
\hline
&$\n^{3,4}$&$(7,8,22,8,7)$&$F_{\rm{I}}$&$F_4$\\
\hline
8&$\n^{8,0}$&$(8,16,30,16,8)$&$E_{\rm{IV}}$&$E_6$\\
\hline
&$\n^{7,1}$&$(8,16,30,16,8)$&$E_{\rm{III}}$&$E_6$\\
\hline
&$\n^{4,4}$&$(8,16,30,16,8)$&$E_{\rm{I}}$&$E_6$\\
\hline
&$\n^{3,5}$&$(8,16,30,16,8)$&$E_{\rm{II}}$&$E_6$\\
\hline
\end{tabular}\label{T9}
\end{table}

We finish our work stating several remarks.

1. The form of the growth vector of $\g$ proposes a conjecture whether $\g$ is simple or not, and then this can be verified rigorously. If it is simple, then one can identify the complex simple Lie algebra such that it is isomorphic to the complexification of $\g$. However, it is not easy to identify the real form of $\g$ related to the pseudo $H$-type Lie algebra $\n$. This requires a careful case by case analysis.

2. If $\dim\n_{-2}$ is equal to $1$ or $2$, then the pseudo $H$-type Lie algebra is related to the contact algebra. We denote by $\rm{ct}(2n+1,\mathbb F)$ the contact algebra over the field $\mathbb F$ of degree~$n$. The contact algebra is a simple 
infinite Lie algebra. It is known that the contact algebra is the prolongation of its negative part which is the Heisenberg algebra~\cite{Ta}. We refer to~\cite{Mo88} for more detailed structure of the contact algebra, its gradation and growth vector to compare the Maple computation to any higher order.

3. If $\dim \n_{-2}$ is equal to 3 or 4,
the growth vector of $\g=\rm{Prol}(\n^{r,s})$ is symmetric for all $(r,s)$ with $r+s = 3$ or $4$
and $\g$ is a simple graded Lie algebra. 
The complexification $\g\otimes \mathbb C$ is determined as in Table~\ref{T8},
or more precisely,
$\big(\frak{sp} (6, \mathbb C), \Sigma_{  \alpha_2 }\big)$,
$\big(\frak{sp} (8, \mathbb C), \Sigma_{\alpha_2}\big)$,
$\big(\frak{sl} (6,\mathbb C), \Sigma_{\alpha_1,\alpha_2}\big)$,
respectively, as graded Lie algebras. 
The real forms corresponding to $\g$
were found by a careful study. An explicit isomorphism between $\n^{1,3} $ and $\mathfrak {su} (3, 3) $ has been described in Section~\ref{sec:example}.
Similar calculations can be performed in other cases.

4. If $\dim \n_{-2} $ is equal to 5 or 6, the computation shows that the prolongation
$\g=\rm{Prol}(\n^{r,s}) $ all vanishes at order 1  for $r+ s = 5, 6$
and therefore is not simple. 
This seems to be a particular case of our general conjecture: {\it for any pseudo $H$-type algebra $\n$,  the prolongation $\g=\rm{Prol(\n)}$ vanishes at order 1 if it does not contain a simple graded Lie algebra of depth 2.}
As a result of computations, combined with the analysis of Table \ref{t:dim},
shows that the negative part of $\frak{so}(2n, \mathbb C) $ cannot be the complexification of
a non-Heisenberg pseudo $H$-type algebra for $n\leq 6$. 

5. In the case $\dim \n_{-2} $  is equal to 7 or 8, we list in Table~\ref{T9} the pseudo $H$-type 
Lie algebras $\n^{r,s} $ with $r+s = 7, 8$ for which the growth vector 
of $\g=\rm{Prol( \n^{r,s}))}$ is symmetric, which means that $\g$ is simple.  
The complexifications $\g \otimes \mathbb C$ are determined 
to be 
$(F_4, \Sigma_{ \alpha_4  } )$ and 
$(E_6, \Sigma{ \alpha_1 , \alpha_4 })$,
respectively,
as it was explained in the proof of Theorem~\ref{th:exH}.
The real simple Lie algebras $\g$, identified through group theoretic observation,
are listed in Table \ref{T9}. For the notation of real forms  in Table~\ref{T9} see~\cite[p. 534]{Helg}.
A detailed study of real forms will be developed in forthcoming research.
\\

{\sc Acknowledgment}.
We thank Ian Anderson for kindly introducing us to the DifferentialGeometry software and the  program for computing the Tanaka prolongation.


\end{document}